\documentclass[9pt,twoside]{extarticle}
\usepackage[utf8]{inputenc}
\usepackage[margin=1.20in,letterpaper]{geometry}

\usepackage{amsmath}
\usepackage{amssymb}
\usepackage{amsthm}
\usepackage{mathrsfs}
\usepackage{float}
\usepackage{xfrac}
\usepackage{enumitem}
\usepackage{bm}
\usepackage[overload]{empheq}

\usepackage{thmtools}
\usepackage{hyperref}
\hypersetup{
    colorlinks=true,
    linkcolor=blue,
    linktocpage=true,
    citecolor=blue,
    pdftitle={cmvp_01},
    bookmarks=true,
    pdfpagemode=FullScreen,
    bookmarksdepth=2,
}

\newcommand{\fin}{\hfill$\square$}

\newcommand{\R}{\bm{R}}

\newcommand{\N}{\bm{N}}

\newcommand{\B}{\bm{B}}

\renewcommand{\phi}{\varphi}
\newcommand{\ep}{\varepsilon}

\renewcommand{\geq}{\geqslant}
\renewcommand{\leq}{\leqslant}

\newcommand{\norm}[1]{\lVert#1\rVert}

\newcommand{\flap}{(-\Delta)^\alpha}

\newcommand{\vara}{\ep^{\frac{2}{3}(1-3\alpha)}}
\newcommand{\varb}{\ep^{\frac{3}{4}\ell}}
\newcommand{\varc}{2}
\newcommand{\vard}{}
\newcommand{\vare}{m}
\newcommand{\varf}{\ep^{\frac{3}{4}\ell}}

\newcommand{\jb}[1]{\langle #1 \rangle}
\newcommand{\COL}[2]{\color{black} #2 \color{black}}

\newcommand{\circI}{\large I}
\newcommand{\circII}{\large II}
\newcommand{\circIII}{\large III}
\newcommand{\circIV}{\large IV}
\newcommand{\circV}{\large V}

\numberwithin{equation}{section}

\newcounter{theorems}[section]

\newtheorem{theorem}{Theorem}[section]
\newtheorem{corollary}{Corollary}[theorem]
\newtheorem{lemma}[theorem]{Lemma}

\renewenvironment{proof}
{
  \vspace{1.5ex}
  \noindent\textbf{Proof.}
}
{
  \fin
  \vspace{1.5ex}
}

\usepackage{fancyhdr}
\newcommand{\Addresses}{{
    \bigskip

  Kyle R. Chickering (Corresponding Author), \textsc{UC Davis
      Department of Mathematics}
  \par\nopagebreak
  \textit{E-mail address:}
  \href{mailto:krc@math.ucdavis.edu}{\texttt{krc@math.ucdavis.edu}}

  \medskip

  Ryan Chris Moreno-Vasquez, \textsc{UC Davis Department of
    Mathematics}
  \par\nopagebreak
  \textit{E-mail address:}
  \href{mailto:rcmorenovasquez@ucdavis.edu}
      {\texttt{rcmorenovasquez@ucdavis.edu}}

  \medskip

  Gavin Pandya, \textsc{UC Davis Department of Mathematics}
  \par\nopagebreak
  \textit{E-mail address:}
  \href{mailto:gpandya@ucdavis.edu}
      {\texttt{gpandya@ucdavis.edu}}
}}

\title{
  \textbf{
    Asymptotically self-similar shock formation for
    1d fractal Burgers equation
  }
}
\author{
  Kyle R. Chickering
  \and
  Ryan C. Moreno-Vasquez
  \and
  Gavin Pandya
}
\date{}

\begin{document}

\maketitle

\begin{abstract}
  For $0<\alpha<\sfrac{1}{3}$ we construct unique solutions to the
  fractal Burgers equation
  \begin{align*}
    \partial_t u + u\partial_xu + \flap u = 0
  \end{align*}
  which develop a first shock in finite time, starting from smooth
  generic initial data. This first singularity is an asymptotically
  self-similar, stable $H^6$ perturbation of a stable, self-similar
  Burgers shock profile. Furthermore, we are able to compute the
  spatio-temporal location and H\"older regularity for the first
  singularity. There are many results showing that gradient blowup
  occurs in finite time for the supercritical range, but the present
  result is the first example where singular solutions have been
  explicitly constructed and so precisely characterized.
\end{abstract}

\tableofcontents

\section{Introduction}
In this paper we study \textit{fractal Burgers equation}
\begin{align}\label{eq:fvbeq}
  \begin{aligned}
  \partial_t u + u\partial_x u + (-\Delta)^{\alpha}u &= 0, \qquad
  \text{on } \R\times \R \\
  u(x,0)&=u_0(x), \qquad \text{on } \R
  \end{aligned}
\end{align}
where the fractional Laplacian $\flap$ is defined as the Fourier multiplier
$|\xi|^{2\alpha}$, and can be represented by the Calderon-Zygmund
singular integral
\begin{align}\label{eq:singular_int}
  ((-\Delta)^\alpha u)(x) = C_\alpha
  \int_{\R}\,\frac{u(x)-u(\eta)}{|x-\eta|^{2\alpha + 1}}\,d\eta
  ,\qquad
  C_\alpha
  = \frac{4^\alpha\Gamma(\alpha + 1/2)}{\pi^{1/2}|\Gamma(-\alpha)|},
\end{align}
taken in the principle value sense. For the remainder of this paper we
will assume all singular integrals are taken in a principle value
sense. The formula \eqref{eq:singular_int} is well known and a
derivation for any dimension is given in \cite{cordoba_2004}.

Equation \eqref{eq:fvbeq} models the competing effects of viscous
regularization and a Burgers-type nonlinearity
\cite{alibaud_2007,kiselev_2008}. The case $\alpha=0$ corresponds to
Burger's equation, which forms shocks in finite time from smooth
initial data, while $\alpha=1$ corresponds to viscous Burger's
equation, whose solutions immediately regularize to $C^\infty$
\cite{whitham_2011}. In particular, \eqref{eq:fvbeq} is a dramatically
simplified model of the interaction between viscous regularization and
nonlinearity in the Navier-Stokes equations.

The equation \eqref{eq:fvbeq} has critical exponent
$\alpha = \sfrac{1}{2}$, at which the orders of the regularization and
the nonlinearity match. It is known that equation \eqref{eq:fvbeq}
forms gradient discontinuities in the supercritical case
$0<\alpha< \sfrac{1}{2}$, and remains globally smooth from smooth data
in both the critical case $\alpha=\sfrac{1}{2}$ and the subcritical
case $\sfrac{1}{2}<\alpha<1$ \cite{kiselev_2008,alibaud_2007}.

\subsection{Historical Study of Fractal Burgers Equation}

Equation \eqref{eq:fvbeq} has been well-studied in the last two
decades. Biler, Funaki, and Woyczynski \cite{biler_1998} established
the existence of self-similar solutions, $L^2$ energy estimates, and
global well-posedness in $H^1$ for the range
$\sfrac{3}{4}<\alpha\leqslant 1$, among other things. More complete
studies of the existence and blowup of solutions to \eqref{eq:fvbeq}
were carried out independently by Alibaud, Droniou, and Vovelle
\cite{alibaud_2007}, and Kiselev, Nazarov, and Shterenberg
\cite{kiselev_2008}. Although they used very different methods, both
proved finite-time blowup in the supercritical case, and global
existence in the critical and subcritical cases. Dong, Du, and Li
\cite{dong_2009} produced an explicit open set of initial data which
form shocks in finite time.  Albritton and Beekie have examined the
long term asymptotic behavior of solutions to \eqref{eq:fvbeq} in the
critical case $\alpha=\sfrac{1}{2}$ and showed that solutions
asymptotically approach a self-similar solution
\cite{albritton_2020}. In this paper, we produce an open set of
initial data which asymptotically approach a self-similar solution to
Burger's equation, and characterize the H\"older regularity of the
shock profile.

As this paper was being submitted to the arXiv
we were made aware of an independent work by Oh and Pasqualotto which
proves the same result using different techniques \cite{oh_2021}.

\subsection{Similarity Variables, Modulation, and Singularity Formation}
The utility of similarity variables has been known for decades, and
the comprehensive study of such methods was started by Barenblatt and
Zel'Dovich \cite{barenblatt_1972,barenblatt_1996} in the 20th
centry. Giga and Kohn \cite{giga_1985,giga_1987} adapted these methods
for studying singularity formation in the nonlinear heat equation, and
made significant contributions to our understanding of the role of
self-similarity in singular behavior. Eggers and Fontelos
\cite{eggers_2008,eggers_2015} have distilled the work of Barenblatt,
Zel'Dovich, Giga, Kohn, and others into a digestible framework
suitable for applications in many physical and analytic contexts. Most
recently, Collot, Ghoul, and Masmoudi \cite{collot_2018} proved that
shocks for Burgers equation are self-similar to leading order.

The present result is inspired by the results obtained by Buckmaster,
Shkoller, and Vicol in their series of papers \cite{buckmaster_2019_1,
  buckmaster_2019_2, buckmaster_2019_3} which prove shock formation
for the 2D isentropic compressible Euler, 3D isentropic compressible
Euler, and 3D non-isentropic compressible Euler equations
respectively. The fundamental idea underlying the method of Buckmaster, Shkoller,
and Vicol is that instabilities in the evolution correspond to
symmetries of the equation under study. Informally speaking, we
introduce ``modulation functions'' which keep track of the equivalence class of the solution as it evolves. 

\subsection{Summary of Present Result}
The present result shows that for a certain range of $\alpha$, we can
view the evolution \eqref{eq:fvbeq} as a perturbation of
\textit{stable} Burgers shock dynamics:

\begin{theorem}
  \textbf{(Imprecise Statement of Result)} Starting from
  non-degenerate initial data, the solution of \eqref{eq:fvbeq} forms a
  generic point shock. Furthermore, we can compute the time, location,
  and regularity of this shock. Furthermore, the shock is an
  asymptotically self-similar Burger's shock.
\end{theorem}
The formal statement of this result is contained in Theorem
\ref{thm:stable_main_thm} and Corollary \ref{cor:open_set}. Our proof
closely follows the strategy laid out in the recent works
\cite{buckmaster_2019_1,buckmaster_2019_2,buckmaster_2019_3,buckmaster_2020,yang_2020},
namely using modulated self-similar variables, careful bounds on
Lagrangian trajectories, and a bootstrapping argument to prove the
result. 

\paragraph*{Symmetry and Modulation.}
To illustrate the need for modulation, suppose we have an exact
initial data $u(\cdot,0)$ which we know leads to a shock, and in the
course of our proof we use the fact that $u(0,0)=0$. We want to show
that in some open neighborhood of $u(\cdot,0)$ a shock will still
form. In our case, we work with the $H^6$ topology, in which a generic
perturbation of $u(\cdot,0)$ no longer fixes the origin.

The reader can check that Fractal Burgers equation \eqref{eq:fvbeq}
is invariant under the following four parameter symmetry group
\begin{align}\label{eq:symms}
  u_*(x,t) = \frac{1}{\lambda}u\left(\frac{x-x_0}{\lambda},
  \frac{t-t_0}{\lambda}\right) + u_0,
\end{align}
where $u$ solves \eqref{eq:fvbeq}. 

Our solution to this problem is to track certain invariants which fix
the equivalence class under the relevant symmetries. For instance,
when we perturb $u(\cdot,0)$ by a sufficiently small amount to
$v(\cdot,0)$, there is a small function $\xi(t)$ such that
$v(\xi(t),t)=0$ for small $t$. By keeping track of $\xi$, we fix the
equivalence class of $v(\cdot,0)$ under spatial translation. The
modulation functions fix the equivalence class of the solution and
allow us to drop the odd symmetry assumptions present in
\cite{alibaud_2007} and prove shock formation on an open set in a
strong topology.

\section{Self-Similar Transformation,
  Initial Data, and Statement of Theorem}
\label{sec:trans_dat_thm}
Our solutions will develop a gradient discontinuity at an
\textit{a priori} unknown time $T_*$ and location $x_*$. Throughout
$\ep$ denotes a small parameter which will be fixed during our proof.

\subsection{The Self-Similar Transformation}\label{sec:transform}
We define the modulation variables
$\tau, \xi, \kappa:[-\ep, T_*]\rightarrow \R$, which respectively
control the temporal location, spatial location, and density of the
developing shock. We fix their initial values at time $t=-\ep$ to be
\begin{align}\label{eq:mod_var_constraints}
  \tau(-\ep)=0, \quad \xi(-\ep)=0, \quad \kappa(-\ep)=\kappa_0.
\end{align}
By their definition, the modulation variables will satisfy
$\tau(T_*)=T_*$ and $\xi(T_*)=x_*$, meaning that the gradient blowup
occurs at the unique fixed point of $\tau$. 

We make the change to \textit{modulated self-similar} variables
described by
\begin{align}\label{eq:self_similar_variables}
  y(x,t) =
  \frac{x-\xi(t)}{(\tau(t)-t)^{3/2}},\qquad s(t) =
  -\log(\tau(t)-t)
\end{align}
and the modulated self-similar variables ansatz 
\begin{align}\label{eq:self_similar_w_ansatz}
  u(x,t)=e^{-\frac{s}{2}}W(y,s) + \kappa(t).
\end{align}
This ansatz corresponds to the stable self-similar scaling for the
stable Burgers profile $\Psi$ (c.f. Appendix \ref{app:burgers}). This
scaling is a natural choice since we are considering the dynamics of
\eqref{eq:fvbeq} as a perturbation from the stable Burgers dynamics.

Plugging the ansatz \eqref{eq:self_similar_w_ansatz} into
\eqref{eq:fvbeq} we obtain the PDE
\begin{align}\label{eq:modulated_pde}
  \left(\partial_s - \frac{1}{2}\right)W + \left(\beta_\tau(W +
  e^{\frac{s}{2}}(\kappa-\dot{\xi})) +
  \frac{3}{2}y\right)\partial_x W
  = -\beta_\tau e^{-\frac{s}{2}}\dot{\kappa} -
  \beta_\tau e^{(3\alpha-1)s} \flap W
\end{align}
for $W$ in $y$ and $s$ and where
\begin{align*}
  \beta_\tau := \frac{1}{1-\dot{\tau}}.
\end{align*}

For convenience we define the transport speed
\begin{align}\label{eq:transport_speed}
  g_W := \beta_\tau\left(W +
  e^{\frac{s}{2}}(\kappa-\dot{\xi})\right) +
  \frac{3}{2}y.
\end{align}
The equations governing the evolution of the derivatives of $W$ are
recorded in Appendix \ref{app:evolution_derivs}.

\subsubsection{Modulation Functions and Constraints on the Evolution}
Imposing the following constraints at $y=0$ fully characterizes the
developing shock, as well as fixing our choice of $\tau$, $\xi$
and $\kappa$:
\begin{align}\label{eq:w_constraints}
  W(0,s)=0, \quad \partial_y W(0,s)=-1, \quad \partial_y^2W(0,s)=0.
\end{align}

For any function $\phi$ we denote $\phi(0,s)=:\phi^0(s)=\phi^0$ out of
convenience. We record the following identities for $\tau$ and $\xi$
\begin{subequations}
  \begin{align}
    \dot{\tau}
    &= -e^{(3\alpha-1)s}\left(\flap \partial_y W\right)^0(s),
      \label{eq:tau_dot_closed}\\
    \kappa - \dot{\xi}
    &=-e^{-s}\dot{\kappa} - e^{(3\alpha-\frac{3}{2})s}\left(
      \flap W\right)^0(s) =
      -\frac{e^{(3\alpha-\frac{3}{2})s}}{\partial_y^3W^0(s)}
      \left(\flap \partial_y^2 W\right)^0(s).
      \label{eq:xi_dot_closed}
  \end{align}
\end{subequations}
We compute \eqref{eq:tau_dot_closed} by evaluating
\eqref{eq:stable_1st} at $y=0$. Equation \eqref{eq:xi_dot_closed} follows from
evaluating \eqref{eq:stable_2nd} and \eqref{eq:modulated_pde} at
$y=0$.

\subsection{Assumptions on the Initial Data}
Let $M>1$ be large, $\ep<1$ small. These constants will be fixed in
the course of the proof and are independent of one another. This
ensures that that we can make quantities of the form $M^p\ep$, $p>0$,
arbitrarily small by taking $\ep$ sufficiently small. We use this
frequently in what follows. We define the constants:
\begin{align}\label{eq:bootstrap_constants}
  m = \frac{3}{16}(1-3\alpha),
  \qquad \ell = \frac{7}{8}(1-3\alpha)
  \qquad q = \min\{\frac{2}{3}(1-3\alpha),
  \frac{1}{6}\},
\end{align}
which are chosen to satisfy certain constraints arising during the
course of the bootstrap argument. We use the notation
$\jb{x}:=(1+x^2)^{1/2}$ throughout and set
\begin{align}\label{eq:s0_h}
  s_0=-\log(\ep) \qquad \text{and} \qquad h=(\log M)^{-2}.
\end{align}
Note that $h$ depends on $M$ but not on $\ep$.

We break $\R$ into three regions. We call $\B_{h}(0)$ the
\textit{Taylor region} since we will rely on Taylor expanding to close
our bootstraps here. The next region is the time dependent annulus
$\B_{e^{ms}}(0) \setminus \B_{h}(0)$ with $m$ defined in
\eqref{eq:bootstrap_constants}. Finally we have the \textit{far field}
which is simply $\R\setminus \B_{e^{ms}}(0)$. Since the fractal
Burgers equation does not propagate compact support, we must maintain
quantitative control on the solution in this region.

\subsubsection{Initial Data in Self-Similar Variables}
We choose initial data $W(0,s)$ in self-similar variables satisfying
the following pointwise equality at the origin:
\begin{align}\label{eq:w_init_origin}
  W^0(s_0)=0, \qquad
  \partial_yW^0(s_0)=-\norm{\partial_yW(\cdot,s_0)}_{L^\infty} = -1,
  \qquad \partial_yW^0(s_0)=0.
\end{align}

We define $\widetilde{W}:=W-\Psi$, where $\Psi$ is the exact,
self-similar Burgers profile \eqref{eq:stable_burg_explicit}. In this
notation we assume
\begin{align}\label{eq:third_x_0}
	|\partial_y^3\widetilde{W}^0(s_0)| \leqslant
	6 \ep^{\frac{16}{9}(1-3\alpha)},
\end{align}

\begin{subequations}\label{eq:w0_init_data}
  \begin{align}[left ={|\widetilde{W}(y,s_0)|\leqslant\empheqlbrace}]
    \ep^{1-3\alpha},
    &\qquad 0 \leqslant |y| \leqslant h \label{eq:w0_init_near} \\
    \frac{1}{2}\ep^q\jb{y}^{\sfrac{1}{3}},
    &\qquad h \leqslant |y| < \infty \label{eq:w0_init_far},
  \end{align}
\end{subequations}
and
\begin{subequations}\label{eq:w1_init_data}
  \begin{align}[left ={|\partial_y\widetilde{W}(y,s_0)|\leqslant\empheqlbrace}]
    \ep^{1-3\alpha},
    &\qquad 0
      \leqslant |y| \leqslant h \label{eq:w1_init_near} \\
    \frac{3}{4}\varf \jb{y}^{-\sfrac{2}{3}},
    &\qquad h
      \leqslant |y| \leqslant \ep^{-m} \label{eq:w1_init_middle} \\
    |\partial_y\widetilde{W}(y,s_0)| \leqslant \varc \ep^2,
    &\qquad \ep^{-m} \leqslant |y| \label{eq:w1_init_far}.
  \end{align}
\end{subequations}
as well as the following bounds on the higher derivatives in the
Taylor region $0 \leqslant |y| \leqslant h$
\begin{align}\label{eq:w234_init_data}
  |\partial_y^2 \widetilde{W}|(y,s_0) \leqslant \ep^{1-3\alpha},
  \qquad
  |\partial_y^3 \widetilde{W}|(y,s_0) \leqslant \ep^{1-3\alpha},
  \qquad
  |\partial_y^4 \widetilde{W}|(y,s_0) \leqslant \ep^{1-3\alpha}.
\end{align}
We will assume the norm bounds
\begin{align}
  \norm{\partial_yW(\cdot,s_0)}_{L^2}\leqslant 200,
  \qquad
  \norm{\partial_y^3W(\cdot,s_0)}_{L^\infty} \leqslant
  \frac{M^{\frac{4}{7}}}{2}, \qquad
  \norm{\partial_y^6W(\cdot,s_0)}_{L^2} \leqslant
  \frac{M^{ \sfrac{1}{2}} }{2}. \label{eq:init_norm_bounds}
\end{align}

\subsubsection{Initial Data in Physical Variables}
We translate \eqref{eq:w_init_origin}-\eqref{eq:init_norm_bounds} into
the physical variable. Our pointwise equalities at the origin become
\begin{align}
  u_0(0)=\kappa_0, \qquad u_0'(0)=-\ep^{-1}=-\norm{u_0'}_{L^\infty},
  \qquad u_0''(0)=0.
\end{align}
Inequality \eqref{eq:third_x_0} becomes
\begin{align}\label{eq:u0_d3_at_0}
  |\partial_\theta^3 u_0(0)-6\ep^{-4}|\leqslant
  6 \ep^{-\sfrac{4}{9}(5+12\alpha)}.
\end{align}
The initial bounds
\eqref{eq:w0_init_data}-\eqref{eq:w1_init_data} in the self-similar
variables become
\begin{subequations}\label{eq:u0_init_data}
  \begin{align}[left ={|u(x,-\ep)|\leqslant\empheqlbrace}]
    \ep^{\frac{1}{2}}\left(\Psi(x\ep^{-\frac{3}{2}}) +
    \ep^{1+3\alpha}\right) + \kappa_0,
    &\qquad 0 \leqslant |x| \leqslant h\ep^{\frac{3}{2}}
      \label{eq:u0_init_near} \\
    \ep^{\frac{1}{2}}\left(\Psi(x\ep^{-\frac{3}{2}}) +
    \frac{1}{2}\ep^q\jb{x\ep^{-\frac{3}{2}}}^{\sfrac{1}{3}}\right)
    +\kappa_0,
    &\qquad h\ep^{\frac{3}{2}} \leqslant |x| < \infty
      \label{eq:u0_init_far}
  \end{align}
\end{subequations}

\begin{subequations}\label{eq:u1_init_data}
  \begin{align}[left ={|\partial_x u(x,-\ep)|
    \leqslant\empheqlbrace}]
    \ep^{-1}\left(\Psi'(x\ep^{-\frac{3}{2}}) + \ep^{1-3\alpha}\right),
    &\qquad 0 \leqslant |x| \leqslant h\ep^{\frac{3}{2}}
    \label{eq:u1_init_near}\\
    \ep^{-1}\left(\Psi'(x\ep^{-\frac{3}{2}}) + \frac{3}{4}\varf
    \jb{x\ep^{-\frac{3}{2}}}^{-\sfrac{2}{3}}\right),
    &\qquad h\ep^{\frac{3}{2}} \leqslant |x|
                                    \leqslant \ep^{-m+\frac{3}{2}}
    \label{eq:u1_init_middle} \\
    \ep^{-1}\left(\Psi'(x\ep^{-\frac{3}{2}}) +
    \frac{1}{2}\ep^{1/2}\varc \ep\right),
    &\qquad \ep^{-m+\frac{3}{2}} \leqslant |x| < \infty
    \label{eq:u1_init_far}
  \end{align}
\end{subequations}
The higher derivative bounds \eqref{eq:w234_init_data} translate to
\begin{align}\label{eq:u234_init_data}
  \begin{split}
    \partial_x^2 u (x,-\ep) \leqslant \ep^{-\sfrac{5}{2}}
    \left(\ep^{1-3\alpha} + \Psi^{(2)}(x \ep^{-\frac{3}{2}})\right), \\
    \partial_x^3 u (x,-\ep) \leqslant \ep^{-4}
    \left(\ep^{1-3\alpha} + \Psi^{(3)}(x \ep^{-\frac{3}{2}})\right), \\
    \partial_x^4 u (x,-\ep) \leqslant \ep^{-\sfrac{11}{2}}
    \left(\ep^{1-3\alpha} + \Psi^{(4)}(x \ep^{-\frac{3}{2}})\right),
  \end{split}
\end{align}
which are valid in the region
$0 \leqslant |x| x\leqslant h\ep^{\frac{3}{2}}$. The norm bounds
\eqref{eq:init_norm_bounds} become
\begin{subequations}
  \begin{align}\label{eq:init_norm_bounds_phys}
    \norm{\partial_x u(\cdot, -\ep)}_{L^2}^2 =
    \ep^{-\frac{1}{2}}\norm{\partial_xW(\cdot,s_0)}_{L^2}^2 \leqslant
    40000\,\ep^{-\frac{1}{2}}, \\
    \norm{\partial_x^3 u(\cdot, -\ep)}_{L^\infty} \leqslant
    \ep^{-4}\norm{\partial_x^3 W_0}_{L^\infty} = \frac{\ep^{-4}}{2}
    M^{\frac{4}{7}}, \label{eq:linf_third_physical} \\
    \norm{\partial_x^6 u(\cdot, -\ep)}_{L^2} =
    \ep^{-4} \norm{\partial_x^6 W_0}_{L^2} \leqslant
    \frac{\ep^{-\frac{31}{4}}}{2}M^{\sfrac{1}{2}}.
    \label{eq:l2_fifth_physical}
  \end{align}
\end{subequations}

\subsection{Main Theorem}
We are now equipped to state the main result of this work

\begin{theorem}\label{thm:stable_main_thm}
  There exist constants $M, \ep > 0$ sufficiently large and small
  respectively such that for some $u_0\in H^6(\R)$, with
  compact support, satisfying
  \eqref{eq:u0_init_data}-\eqref{eq:init_norm_bounds_phys} at
  $t=-\ep$, the unique solution to \eqref{eq:fvbeq} forms a shock in
  finite time. Furthermore
  \begin{enumerate}[label=(\roman*)]
  \item For any $\overline{T}<T_*$,
    $u\in C([-\ep,\overline{T}];C^5(\R))$.
  \item The blowup location $\xi(T_*)=x_*$ is unique.
  \item The blowup time $T_*$ occurs prior to
    $t=\frac{3}{4}\ep^{\frac{8}{9}(1-3\alpha)}$ and the blowup
    location $x_*$ satisfies $|x_*|\leqslant 4M\ep$. The
    blowup time and location are explicitly computable.
  \item $\lim_{t\rightarrow T_*}\partial_\theta u = -\infty$ and
    furthermore we have the precise control
    \begin{align*}
      \frac{1}{2(T_*-t)} \leqslant |\partial_\theta
      u(\xi(t),t)|=\norm{\partial_\theta u}_{L^\infty} \leqslant
      \frac{1}{T_*-t}.
    \end{align*}
  \item The solution $u$ converges asymptotically in the self-similar
    space to a self-similar solution to the Burgers equation
    $\Psi_\nu$ given by \eqref{eq:psi_nu}. I.e.
    \begin{align*}
      \limsup_{s\rightarrow \infty}\norm{W-\Psi_\nu}_{L^\infty} = 0.
    \end{align*}
  \item The solution has H\"older regularity $\sfrac{1}{3}$ at the
    blowup time, i.e. $u(x,T_*)\in C^{1/3}(\R)$.
  \end{enumerate}
\end{theorem}

An example of an initial data satisfying these conditions and leading
to a shock in finite time is $\ep^{1/2}\eta(x)\Psi(x\ep^{-3/2})$ for a
suitable compactly supported smooth function $\eta$.

The gradient blowup outlined in Theorem \ref{thm:stable_main_thm} is
stable and we can relax the initial conditions above and still satisfy
the conclusions of Theorem \ref{thm:stable_main_thm}:

\begin{corollary}\label{cor:open_set}
  There exists an open set of data in the $H^6$ topology such that the
  conclusions of Theorem \ref{thm:stable_main_thm} still hold.
\end{corollary}

We prove Theorem \eqref{thm:stable_main_thm} in section
\ref{sec:proof_of_stable_main} and the corollary \ref{cor:open_set} in
section \ref{sec:open_set}.

\section{Bootstraps: Assumptions and Consequences}
\label{sec:bootstraps}
Throughout we use $\lesssim$ to denote a quantity which is bounded by
a constant depending on $\alpha$ or $h$, but not on $s,\ep$, or
$M$. We establish our bootstrap assumptions in
\ref{sec:bootstrap_assumptions}, derive consequences from these
assumptions in \ref{sec:consequences}, close our top order energy
estimate in \ref{sec:top_order_energy}, and establish control on the
Lagrangian trajectories in \ref{sec:lagrangian}. In Section
\ref{sec:boot_linf} we close our bootstraps in $L^\infty$.

\subsection{Bootstrap Assumptions}\label{sec:bootstrap_assumptions}
Recall that $\widetilde{W}=W-\Psi$, where $\Phi$ is the stable Burgers
profile \eqref{eq:stable_burg_explicit}. We assume the bootstraps
\begin{subequations}\label{eq:w0_boot}
  \begin{align}[left ={|\widetilde{W}|(y,s)\leqslant\empheqlbrace}]
    \left(\ep^{\frac{8}{9}(1-3\alpha)}
    + \log M \ep^{\frac{7}{9}(1-3\alpha)} \right)h^4,
    &\qquad 0 \leqslant |y| \leqslant h \label{eq:w0_boot_near} \\
    \ep^q\jb{y}^{\sfrac{1}{3}},
    &\qquad h \leqslant |y| < \infty \label{eq:w0_boot_far}
  \end{align}
\end{subequations}

\begin{subequations}\label{eq:w1_boot}
  \begin{align}[left ={|\partial_y\widetilde{W}|(y,s)
    \leqslant\empheqlbrace}]
    \left(\ep^{\frac{8}{9}(1-3\alpha)}
    + \log M \ep^{\frac{7}{9}(1-3\alpha)} \right)h^3,
    &\qquad 0 \leqslant |y| \leqslant h \label{eq:w1_boot_near} \\
    \varb \jb{y}^{-\sfrac{2}{3}},
    &\qquad h \leqslant |y| \leqslant e^{ms}
    \label{eq:w1_boot_middle} \\
    2 e^{-s},
    &\qquad e^{ms} \leqslant |y|\label{eq:w1_boot_far}
\end{align}
\end{subequations}
on the solution and the first derivative. For the higher derivatives
we assume the following bootstraps on the Taylor region
$0 \leqslant |y| \leqslant h$
\begin{subequations}\label{eq:w234_boot_near}
  \begin{align}
    |\partial_y^2 \widetilde{W}|(y,s) \leqslant
    \left( \ep^{\frac{8}{9}(1-3\alpha)}
    + \log M \ep^{\frac{7}{9}(1-3\alpha)} \right)h^2,
    \label{eq:w2_boot_near} \\
    |\partial_y^3 \widetilde{W}|(y,s) \leqslant
    \left( \ep^{\frac{8}{9}(1-3\alpha)}
    + \log M \ep^{\frac{7}{9}(1-3\alpha)} \right)h,
    \label{eq:w3_boot_near} \\
    |\partial_y^4 \widetilde{W}|(y,s)
    \leqslant \ep^{\frac{7}{9}(1-3\alpha)}, \label{eq:w4_boot_near}
\end{align}
\end{subequations}
for $h$ given by \eqref{eq:s0_h} and $m,\ell,q$ given by
\eqref{eq:bootstrap_constants}.

To compensate for a loss of derivatives, we assume the following
$L^\infty$ control over the third derivative
\begin{align}\label{eq:w3_global}
  \norm{\partial_y^3W}_{L^\infty} \leqslant M^{\vard}.
\end{align}
Finally, we assume the precise bootstrap
\begin{align}\label{eq:w3_boot}
  |\partial_y^3W(0,s)-6|\leqslant 10\ep^{\frac{16}{9}(1-3\alpha)} < 1
\end{align}
which will ensure that our solution converges to the correct
profile. For the modulation variables we assume
\begin{subequations}\label{eq:bootstrap_assumptions}
  \begin{align}
    |\tau(t)|
    &\leqslant \ep^{\frac{8}{9}(1-3\alpha)},
    \quad |\xi(t)|\leqslant 4M\ep, \label{eq:mod_boot_assump} \\
    |\dot{\tau}(t)|
    &\leqslant \ep^{\frac{8}{9}(1-3\alpha)},
      \quad |\dot{\xi}(t)|\leqslant
      5M \label{eq:mod_deriv_boot_assump}
  \end{align}
\end{subequations}
for all $t<T_*$.

\subsection{Consequences of Bootstraps}\label{sec:consequences}
The following are immediate consequences of the bootstraps
\eqref{eq:w0_boot}-\eqref{eq:bootstrap_assumptions}.
\begin{enumerate}
\item The explicit property \eqref{eq:psi_jb} of $\Psi'$ along with
  \eqref{eq:w1_boot} shows that $\partial_yW\in L^\infty_{y,s}(\R)$
  with the following uniform bound
  \begin{align}\label{eq:w1_unif_linf}
    \norm{\partial_yW}_{L^\infty_{y,s}}= 2.
  \end{align}

\item From the bootstrap \eqref{eq:mod_deriv_boot_assump} on
  $\dot{\tau}$ we have the following bound on $\beta_\tau$
  \begin{align}\label{eq:beta_tau_bound}
    \beta_\tau = \frac{1}{1-\dot{\tau}} \leqslant \frac{1}{1 -
    \ep^{\frac{1}{2}(1-3\alpha)}} \leqslant 1 +
    2\ep^{\frac{1}{2}(1-3\alpha)} \leqslant \frac{3}{2}.
  \end{align}
  This bound is true for $0\leqslant \ep < \frac{1}{8}$,
  and we will end up taking $\ep$ much smaller than this.

\item We have the following bound on $\kappa$:
  \begin{align}\label{eq:kappa_bound}
    |\kappa(t)|=|u(\xi(t),t)| \leqslant
    \norm{u}_{L^\infty_{x,t}}\leqslant M.
  \end{align}
  This is shown with the following argument. Equation \eqref{eq:fvbeq}
  satisfies a trivial maximum principle: if
  $\norm{u_0}_{L^\infty}\leqslant M$ than
  $\norm{u}_{L_{\theta,t}^\infty}\leqslant M$. This is an easy
  consequence of $\flap u (x_0,t_0)$ being positive definite at a
  maximum. Then observe that $\kappa(t)=u(\xi(t),t)$ by
  \eqref{eq:self_similar_w_ansatz} and the constraint
  \eqref{eq:w_constraints}.

\item The transformation \eqref{eq:self_similar_w_ansatz} and the
  bound \eqref{eq:kappa_bound} on $\kappa$ give
  \begin{align}\label{eq:lagrangian_est_1}
  \norm{W}_{L^\infty_y} \leqslant e^{s/2}(\norm{u_0}_{L^\infty} + M).
\end{align}
\end{enumerate}

\noindent The bootstrap \eqref{eq:w1_boot} is insufficient to
guarantee that $\partial_yW\in L^2$; we must use the structure of the
equation \eqref{eq:modulated_pde}.

\begin{lemma}\label{lem:uniform_w_1_l2}
  \textbf{(Uniform $L^2$ Bound for $\partial_yW$).}  The bootstraps
  \eqref{eq:w0_boot}-\eqref{eq:w1_boot_far} imply that
  $\partial_y W$ is uniformly $L^2(\R)$ in self-similar time.
\end{lemma}

\begin{proof}
  $\partial_yW$ solves \eqref{eq:stable_1st}, which we recall for
  convenience
  \begin{align*}
    \bigg(\partial_s +1 + \beta_\tau \partial_y W\bigg)\partial_yW
    + g_W\partial^2_y W
    =  - \beta_\tau e^{(3\alpha-1)s} \flap \partial_yW.
  \end{align*}
  Taking the $L^2$ inner product with \eqref{eq:stable_1st} gives
  \begin{align}\label{eq:w1_l2}
    \begin{aligned}
      \frac{1}{2}\frac{d}{ds}\norm{\partial_y W}_{L^2}^2
      + \norm{\partial_y W}_{L^2}^2 +
      \int_{\R}\,g_W\partial_y W\partial_y^2W\,dy
      &+ \beta_\tau \int_{\R}\,(\partial_y W)^3\,dy = \\
      &\quad -\beta_\tau
      e^{(3\alpha-1)s}(\flap \partial_y W, \partial_y W)_{L^2}.
    \end{aligned}
  \end{align}
  $\flap$ is self-adjoint on $L^2$ and thus we have
  $(\flap \partial_yW, \partial_yW)_{L^2} \geqslant 0$. The transport
  term satisfies
  \begin{align}\label{eq:w1_energy_1}
    \begin{split}
      \int_{\R}\,g_w\partial_y W\partial_y^2W\,dy &= - \frac{1}{2}
      \int_{\R}\,g_W\frac{d}{dy}(\partial_y W)^2\,dy \\
      &= - \frac{\beta_\tau}{2}\int_{\R}\,(\partial_y W)^3\,dy -
      \frac{3}{4}\norm{\partial_y W}_{L^2}^2.
    \end{split}
  \end{align}
  From the uniform $L^\infty$ bound \eqref{eq:w1_unif_linf} and the
  bootstraps \eqref{eq:w1_boot_middle}, \eqref{eq:w1_boot_far} it
  follows that
  \begin{align}\label{eq:w1_energy_2}
    \begin{split}
      \int_{\R}\,(\partial_yW)^3\,dy
      &= \left(\int_{0\leqslant |y|\leqslant h}
        + \int_{h \leqslant |y| \leqslant e^{ms}}
        + \int_{e^{ms} \leqslant |y|<\infty}\right)\,
      (\partial_yW)^3\,dy \\
      &\leqslant C + \ep^{\frac{9}{4}\ell}
      \int_{h \leqslant |y| \leqslant e^{ms}}\,
      \frac{1}{1+y^2}\,dy + \varc
      \int_{e^{ms}\leqslant |y| < \infty}\,
      e^{-s}|\partial_yW|^2\,dy\\
      &\leqslant C + 2\arctan(e^{ms}) +
      \varc e^{-s}\norm{\partial_y^2 W}_{L^2}^2 \\
    &\leqslant C + \varc \ep\norm{\partial_yW}_{L^2}^2
    \end{split}
  \end{align}
  where $C$ is a constant which changes from line to line.

  Combining \eqref{eq:w1_l2} with
  \eqref{eq:w1_energy_1}-\eqref{eq:w1_energy_2} gives us
  \begin{align*}
     \frac{d}{ds}\norm{\partial_xW}_{L^2}^2
    &= \frac{C\beta_\tau}{2} +
      \left(
      \beta_\tau\varc\ep^b-\frac{1}{4}
      \right)\norm{\partial_xW}_{L^2}^2 \\
    &\leqslant
      \frac{C\beta_\tau}{2} 
    -\frac{1}{8}\norm{\partial_xW}_{L^2}^2
  \end{align*}
  where the last inequality is obtained by taking $\ep$ sufficiently
  small and estimating $\beta_\tau$ using \eqref{eq:beta_tau_bound}.
  Use Gr\"onwall's inequality with the initial data assumption
  \eqref{eq:init_norm_bounds} to conclude.
\end{proof}

\subsection{Closure of Top Order Energy
  Estimate}\label{sec:top_order_energy}
The equation \eqref{eq:fvbeq} suffers an $L^\infty$ loss of
derivatives; we cannot estimate $\partial_x^ku$ in $L^\infty$ without
control of $\partial_x^{k+1}u$ in $L^\infty$. To overcome
this difficulty we use our bootstraps to linearize the evolution of
the sixth derivative in $L^2$ and close the bootstrap
\eqref{eq:w3_global} on the third derivative by interpolation (see
Lemma \ref{lem:gn_interpolation}).

The evolution of $\partial_y^6W$ is given by \eqref{eq:stable_6th}
which we reproduce for convenience
\begin{align*}
  \bigg(\partial_s + \frac{11}{2}
  + 7 \beta_\tau\partial_yW\bigg)\partial_y^6W
  + g_W\partial_y^7W
  = -\beta_\tau\left[e^{(3\alpha-1)s} \flap \partial_y^6W
  + 35\partial_y^3W\partial_y^4W + 21\partial_y^2W \partial_y^5 W
  \right].
\end{align*}
We test in $L^2$ against $\partial_y^6W$
\begin{align*}
  \frac{1}{2} \frac{d}{ds}\norm{\partial_y^6W}_{L^2}^2 
  + \frac{17}{2} \norm{\partial_y^6W}_{L^2}^2
  &= -\COL{red}{\underbrace{
    \int_{\R}\,g_W\partial_y^7W\partial_y^6W\,dy}_{\circI}}
    - 7\beta_\tau\int_{\R}\,\partial_yW(\partial_y^6W)^2\,dy \\
  &\quad- 35\beta_\tau\int_{\R}\,
    \partial_y^3W\partial_y^4W\partial_y^6W\,dy
    -\COL{orange}{\underbrace{
    21\beta_\tau\int_{\R}\,\partial_y^2W \partial^5 W\partial_y^6W\,dy
    }_{\circII}} \\
 &\quad-\COL{blue}{\underbrace{\beta_\tau e^{(3\alpha-1)s}\int_{\R}\,
    \partial_y^6W\flap \partial_y^6W\,dy}_{\circIII}}. \\
\end{align*}
The operator $\flap$ is self-adjoint on $L^2$, allowing us to drop
$\COL{blue}{\circIII}$ from subsequent estimates. Integration by parts
and the supposed $L^2$ decay of $\partial_y^6W$ gives
\begin{align*}
  \COL{red}{\circI}
  &= \frac{1}{2}\int_{\R}\,g_W\frac{\partial}{\partial
    y}(\partial_y^6W)^2\,dy \\
  &=\frac{\beta_\tau}{2}\left(
    \int_{\R}\,W\frac{\partial}{\partial y}(\partial_y^6W)^2\,dy 
    + e^{\sfrac{s}{2}}(\kappa-\dot{\xi})
    \int_{\R}\,\frac{\partial}{\partial y}(\partial_y^6W)^2\,dy\right)
    + \frac{3}{4}\int_{\R}
    y\frac{\partial}{\partial y}(\partial_y^6W)^2\,dy \\
  &=-\frac{\beta_\tau}{2}\int_{\R}\,\partial_yW(\partial_y^6W)\,dy -
    \frac{3}{4}\norm{\partial_y^6W}_{L^2}^2.
\end{align*}
Similarly, we form a total derivative and integrate by parts to obtain
\begin{align*}
  \COL{orange}{\circII} =
  \frac{21}{2} \beta_\tau\int_{\R}\,\partial_y^3W(\partial_y^5W)^2\,dy.
\end{align*}
Collecting like terms and using the bootstrap assumption
\eqref{eq:w3_global} to bound the third derivative in $L^\infty$,
we obtain the energy equality
\begin{align}\label{eq:stable_energy}
  \begin{aligned}
    \frac{d}{ds}\norm{\partial_y^6W}_{L^2}^2 &=
    -\frac{31}{2}\norm{\partial_y^6W}_{L^2}^2
    -13 \beta_\tau
    \int_{\R}\,\partial_yW(\partial_y^6W)^2\,dy\\
    & \qquad 
    -70 \beta_\tau
    \int_{\R}\,\partial_y^3W\partial_y^5W \partial_y^6 W\,dy
    + 21 \beta_\tau \int_{\R}\,\partial_y^3W(\partial_y^5W)^2\,dy
    - 2\norm{\partial_y^6W}_{H^{\alpha/2}} \\
    & \leqslant
    \left(-\frac{31}{2} + 13 \beta_\tau \right)
    \norm{\partial_y^6W}_{L^2}^2 +
    35\beta_\tau M \norm{\partial_y^4 W}_{L^2}
    \norm{\partial_y^6 W}_{L^2}
    + 21 \beta_\tau M \norm{\partial_y^5 W}_{L^2}^2.
    \end{aligned}
\end{align}

We use Galiardo-Nirenberg to interpolate between $\partial_y^6W$ in
$L^2$ and the uniform $L^2$
bound of $\partial_yW$ proven in Lemma \ref{lem:uniform_w_1_l2}:
\begin{align*}
  \norm{\partial_y^4W}_{L^2} \lesssim 
  \norm{\partial_yW}_{L^2}^{\sfrac{2}{5}}
  \norm{\partial_y^6W}_{L^2}^{\sfrac{3}{5}}
  \lesssim
  \norm{\partial_y^6W}_{L^2}^{\sfrac{3}{5}}, \\
  \norm{\partial_y^5W}_{L^2} \lesssim 
  \norm{\partial_yW}_{L^2}^{\sfrac{1}{5}}
  \norm{\partial_y^6W}_{L^2}^{\sfrac{4}{5}}
  \lesssim
  \norm{\partial_y^6W}_{L^2}^{\sfrac{4}{5}}.
\end{align*}

We insert the interpolants into our energy estimate
\eqref{eq:stable_energy}, take $\beta_\tau$ close to 1 via
\eqref{eq:beta_tau_bound}, and apply Young's inequality to find
\begin{align*}\label{eq:stable_energy_final}
  \frac{d}{ds}\norm{\partial_y^6W}_{L^2}^2 \lesssim
  \left( -\frac{31}{2} + 13 \beta_\tau \right)\norm{\partial_y^6W}_{L^2}^2
  +  M \beta_\tau \norm{ \partial_y^6 W}_{L^2}^{\sfrac{8}{5}}
  \lesssim M -
  2\norm{\partial_y^6W}_{L^2}^2.
\end{align*}
Integrating this differential inequality gives
\begin{align}
 \norm{\partial_y^6W}_{L^2}^2 \lesssim
  \left(1-e^{-2c(s-s_0)}\right)
  M + \norm{\partial_y^6W^0}_{L^2}^2e^{-2c(s-s_0)}.
\end{align}
Using our initial data assumption \eqref{eq:init_norm_bounds} and
taking $M$ sufficiently large, we obtain
\begin{align}\label{eq:energy_estimate}
	\norm{\partial_y^6W}_{L^2}^2 \lesssim M \leq M^{\sfrac{3}{2}}.
\end{align}

Finally, by Gagliardo-Nirenberg
\begin{align}\label{eq:3rd_dv_linf}
  \norm{\partial_y^3W}_{L^\infty} \lesssim
  \norm{\partial_yW}_{L^\infty}^{\sfrac{5}{9}}
  \norm{\partial_y^6 W}_{L^2}^{\sfrac{4}{9} } \lesssim
  M^{\sfrac{1}{3}\vard}.
\end{align}
This closes the bootstrap \eqref{eq:w3_global} upon choosing $M$
sufficiently large.

\paragraph{$L^\infty$ Interpolation of the Intermediary
  Derivatives.}
We avoid bootstrapping the global bounds on $\partial_y^2W$ and
$\partial_y^3W$ by interpolation. Gagliardo-Nirenberg gives the
following bounds
\begin{subequations}
  \begin{align}
    \norm{\partial_y^5W}_{L^\infty} \lesssim
    \norm{\partial_yW}_{L^\infty}^{1/9}
    \norm{\partial_y^5W}_{L^2}^{8/9}
    \lesssim M^{\frac{2}{3}}
    \label{eq:5th_dv_linf} \\
    \norm{\partial_y^4W}_{L^\infty} \lesssim
    \norm{\partial_yW}_{L^\infty}^{1/3}
    \norm{\partial_y^5W}_{L^2}^{2/3}
    \lesssim M^{\frac{1}{2}}
    \label{eq:4th_dv_linf} \\
    \norm{\partial_y^2W}_{L^\infty} \lesssim
    \norm{\partial_yW}_{L^\infty}^{7/9}
    \norm{\partial_y^5W}_{L^2}^{2/9}
    \lesssim M^{\frac{1}{6}}
    \label{eq:2nd_dv_linf}
  \end{align}
\end{subequations}

With control on $\partial_y^2 W$ and $\partial_y^3W$ in $L^\infty$
from \eqref{eq:2nd_dv_linf} and \eqref{eq:3rd_dv_linf}, we
can combine the identity \eqref{eq:xi_dot_closed} and interpolation
inequality \eqref{eq:interpolation} for $\flap$ to obtain
\begin{align}\label{eq:lagrangian_est_2}
  e^{\sfrac{s}{2}}|\kappa -\dot{\xi}| \lesssim
  \frac{e^{(3\alpha-1)s}}{5}\norm{\partial_x^2
  W}_{L^\infty}^{1-2\alpha}\norm{\partial_x^3 W}_{L^\infty}^{2\alpha}
  \leqslant e^{\frac{9}{10}(3\alpha-1)s},
\end{align}
sacrificing powers of $\ep$ to absorb the constant. 

\subsection{Lagrangian Trajectories}\label{sec:lagrangian}
We prove upper and lower bounds on the Lagrangian trajectories of
\eqref{eq:modulated_pde} which allow us to convert spatial decay into
temporal decay.

The flowmap evolution of \eqref{eq:modulated_pde} is
\begin{align}
  \frac{d}{ds}\Phi^{y_0} = \frac{3}{2}\Phi^{y_0} +
  \beta_\tau\left(W^{y_0} + e^{\frac{s}{2}}(\kappa-\dot{\xi})\right),
  \label{eq:lagrange_traj}
\end{align}
where $\Phi^{y_0}$ is the trajectory emanating from the label $y_0$
and $W^{y_0}(s):=W(\Phi^{y_0}(s),s)$.

\begin{lemma}\label{lem:optimal_upper}
  \textbf{(Upper Bound on Trajectories).}
  Let $x_0\in \R$, then trajectories are bounded above by
  \begin{align}\label{eq:lagrangian_upper}
      |\Phi^{x_0}(s)|
      &\leqslant \left(|x_0| +
        \frac{3}{2}C\ep^{-\sfrac{1}{2}}\right)e^{\frac{3}{2}(s-s_0)},
  \end{align}
  with $C$ a positive constant depending linearly on $M$.
\end{lemma}

\begin{proof}
  The flowmap equation \eqref{eq:lagrange_traj} implies that
  \begin{align*}
    \frac{d}{ds}\left[\Phi^{x_0}e^{-3s/2}\right] = \beta_\tau
    e^{-3s/2} \left(W^{x_0} + e^{s/2}(\kappa-\dot{\xi})\right).
  \end{align*}
  We integrate in time, estimating $\beta_\tau$ with
  \eqref{eq:beta_tau_bound}, and apply \eqref{eq:lagrangian_est_1},
  \eqref{eq:lagrangian_est_2} to obtain
  \begin{align*}
    |\Phi^{x_0}(s)|
    &\leqslant |x_0|e^{\frac{3}{2}(s-s_0)} +
      e^{3s/2}\int_{s_0}^s\,e^{-\frac{3}{2}s'}\beta_\tau
      \left(W^{x_0}(s') +
      e^{s'/2}(\kappa-\dot{\xi})\right)\,ds' \\
    &\leqslant \ep|x_0|e^{\frac{3}{2}s} +
      \frac{3}{2}\ep^{\frac{3}{2}s}\int_{s_0}^s\,
      e^{-s'}\left(\norm{u_0}_{L^\infty} + 6M\right) +
      e^{-\frac{3}{2}s'}\ep^{\frac{9}{10}(1-3\alpha)}\,ds',
  \end{align*}
  from which the result follows.
\end{proof}

Next we prove two lower bounds on the Lagrangian trajectories. The
first bound serves only to guarantee that all trajectories emanating
from outside the Taylor region will eventually take off with our
optimal bound \eqref{eq:optimal_lower}.

\begin{lemma}\label{lem:sub_optimal_lemma}
  Let $h \leqslant |x_0|<\infty$. The trajectories are bounded below
  by
  \begin{align}\label{eq:sub_optimal_lower}
    |\Phi^{x_0}(s)| \geqslant |x_0|e^{\frac{1}{5}(s-s_0)}.
  \end{align}
\end{lemma}

\begin{proof}
  The proof is contained in \cite{yang_2020}, replacing the author's
  bound for $e^{\frac{s}{2}}(\kappa-\dot{\xi})$ by our bound
  \eqref{eq:lagrangian_est_2} and taking $\ep$ sufficiently small.
\end{proof}

To close the bootstraps, this $\sfrac{1}{5}$ bound proves
insufficient; we require a stronger lower bound to prove Lemma
\ref{lem:forcing_decay_far} and to close the bootstrap
\eqref{eq:w0_boot_far}. In particular we must have that the
trajectories eventually take off faster than $e^{\frac{3}{5}s}$. The
following estimate is optimal in light of Lemma
\ref{lem:optimal_upper} and the transformation
\eqref{eq:self_similar_variables}.

\begin{lemma}\label{lem:optimal_lemma}
  Let $1 \leqslant |x_0|< \infty$. The trajectories are bounded below
  by
  \begin{align}\label{eq:optimal_lower}
    |\Phi^{x_0}(s)| \geqslant \left(|x_0|^{2/3}-\frac{2
    C}{3}\right) e^{\frac{3}{2}(s-s_0)} > 0.
  \end{align}
  The constant $C$ is
  \begin{align*}
    C=(1+2\ep^{\frac{1}{2}(1-3\alpha)})\left(1 +
    \frac{1}{h^{1/3}}\left(\ep^q\jb{h}+
    \ep^{\frac{1}{2}(1-3\alpha)}\right)\right).
  \end{align*}
\end{lemma}

\begin{proof}
  Multiply the flowmap \eqref{eq:lagrange_traj} by $\Phi^{x_0}$ to
  obtain
  \begin{align*}
    \frac{d}{ds}|\Phi^{x_0}|^2 = 3|\Phi^{x_0}|^2 +
    2\beta_\tau\Phi\left(W + e^{\frac{s}{2}}(\kappa-\dot{\xi})\right).
  \end{align*}
  Recall that $\Psi$ is the self-similar solution to Burgers equation
  \eqref{eq:burgers_similarity_equation} and observe that in light of \eqref{eq:psi_0_13} and
  \eqref{eq:w0_boot_far} we have
  \begin{align*}
    |W(x,s)|
    \leqslant |\Psi(x)| + |\widetilde{W}(x,s)|
    \leqslant |x|^{1/3} + \ep^q\jb{x}^{1/3}
    \leqslant \left(1 +
      \ep^{q}\frac{\jb{h}^{1/3}}{h^{1/3}}\right)|x|^{1/3}.
  \end{align*}
  In a similar vein, \eqref{eq:lagrangian_est_2} implies that
  \begin{align*}
    |e^{\frac{s}{2}}(\kappa-\dot{\xi})| \leqslant
    \frac{\ep^{\frac{1}{2}(1-3\alpha)}}{h^{1/3}}|x^{1/3}|.
  \end{align*}
  Estimating $\beta_\tau$ as in \eqref{eq:beta_tau_bound} and taking
  $C$ as in the statement of the lemma, we obtain the differential
  inequality
  \begin{align*}
    \frac{d}{ds}|\Phi^{x_0}|^2 - 3|\Phi^{x_0}|^2 \geqslant -
    2C|\Phi^{x_0}|^{4/3}.
  \end{align*}
  This inequality is separable in the variable
  $|\Phi^{x_0}|^2e^{-3s}$, and integrating from $s_*$,
  $s_0\leqslant s_* < s$ yields \eqref{eq:optimal_lower} upon taking
  $\ep$ sufficiently small.
\end{proof}

Together, these two lemmas show that all trajectories with starting
labels $h \leqslant |x_0|$ will eventually take off like
$e^{\frac{3}{2}s}$, i.e. $|\Phi^{x_0}(s)|\lesssim e^{\frac{3}{2}s}$
whenever $h\leqslant |x_0|$.

\section{Bootstraps: Closing The $L^\infty$ Estimates}
\label{sec:boot_linf}
\subsection{Modulation Variables}
Our interpolation inequality \eqref{eq:interpolation} for $\flap$ and
the estimates \eqref{eq:w1_unif_linf}, \eqref{eq:2nd_dv_linf} give
\begin{align}\label{eq:tau_dot_decay}
  |\dot{\tau}|
  \leqslant e^{(3\alpha-1)s}\norm{\partial_yW}_{L^\infty}^{1-2\alpha}
  \norm{\partial_y^2 W}_{L^\infty}^{2\alpha}
  \lesssim M^{\frac{1}{3}\alpha}e^{(3\alpha - 1)s}
  \leqslant \frac{1}{2}e^{\frac{8}{9}(3\alpha-1)s}
\end{align}
upon choosing $\ep$ sufficiently small. We apply the fundamental
Theorem of calculus and recall that $\tau(-\ep)=0$ to find
\begin{align*}
  |\tau(t)| \leqslant \int_{-\ep}^t\,|\dot{\tau}(t')|\,dt'
  \leqslant \frac{1}{2}\ep^{\frac{8}{9}(1-3\alpha)}(T_*+\ep)
  \leqslant \frac{1}{2}\ep^{1-3\alpha} + \frac{1}{2}\ep^{1 +
  \frac{1}{2}(1-3\alpha)}
  \leqslant \frac{3}{4}\ep^{\frac{8}{9}(1-3\alpha)}.
\end{align*}
This closes the $\tau$ bootstraps \eqref{eq:mod_boot_assump} and
\eqref{eq:mod_deriv_boot_assump}.

The $\xi$ bootstrap is equally simple. By the identity
\eqref{eq:xi_dot_closed},
\begin{align}\label{eq:xi_estimate}
  \begin{aligned}
  |\dot{\xi}|
  &\leqslant |\kappa| +
  \frac{e^{s(3\alpha-\frac{3}{2})}}{|\partial_y^3W(0,s)|}(\flap
    \partial_y^2 W)(0,s) \\
  &\leqslant M +
    \frac{1}{5}e^{(3\alpha-\frac{3}{2})s}M^{\frac{1}{6}+\frac{\alpha}{3}}
  \\
  &\leqslant 2M,
  \end{aligned}
\end{align}
where we have chosen $\ep$ sufficiently small and $M$ sufficiently
large. Integrating \eqref{eq:xi_estimate} in time and using the
fundamental theorem of calculus, recalling that $|\xi(-\ep)|=0$, 
\begin{align*}
  |\xi(t)| \leqslant \int_{-\ep}^t \,|\dot{\xi}|\,dt' \leqslant
  2M(T_*+\ep) \leqslant 2M(\ep^{1+2\alpha}+\ep) \leqslant 4M\ep.
\end{align*}
This closes the two $\xi$ bootstraps
\eqref{eq:mod_boot_assump} and \eqref{eq:mod_deriv_boot_assump}.

\subsection{Taylor Region}
\paragraph*{(Fourth Derivative on $0 \leqslant |y|\leqslant h$).}
The fourth derivative of $\widetilde{W}$ evolves
according to \eqref{eq:diff_4_deriv}, which we recall here
\begin{align*}
  \big(\partial_s + \underbrace{\frac{11}{2} +
    5\beta_\tau \partial_yW}_{\mathcal{D}(y)}\big)
  \partial_y^4\widetilde{W}
  + g_W\partial_y^4\widetilde{W}
  &= - \beta_\tau\mathcal{F}.
\end{align*}
$\mathcal{D}$ is the damping and $\mathcal{F}$ is the forcing,
which is given explicitly in \eqref{eq:fourth_diff_force}.

Using \eqref{eq:w1_unif_linf} and \eqref{eq:beta_tau_bound} we bound
the damping below by
\begin{align}\label{eq:fourth_taylor_damping}
  \mathcal{D}
  \geqslant \frac{11}{2} - 5\beta_\tau \geqslant \frac{1}{4}.
\end{align}

The forcing is given by
\begin{align}\label{eq:fourth_diff_force}
  \begin{split}
  \mathcal{F}
  :=- 
  \underbrace{\beta_\tau e^{(3\alpha-1)s}\flap \partial_y^4W}_{\circI}
  &+ \underbrace{\beta_\tau\left(\widetilde{W}\Psi^{(5)}
    + 4\partial_y\widetilde{W}\Psi^{(4)}
    + 8\partial_y^2\widetilde{W}\Psi^{(3)}
    + 10 \partial_y^3\widetilde{W}\Psi''
    + 11\partial_y^2\widetilde{W}\partial_y^3\widetilde{W}
  \right)}_{\circII} \\
&\qquad+ \underbrace{
  \beta_\tau e^{\sfrac{s}{2}} (\kappa -\dot\xi)\partial_y^5\Psi}_{\circIII}
+ \underbrace{\beta_\tau\dot{\tau}\left(\Psi\partial_y^5\Psi
    + 5\Psi'\Psi^{(4)} + 10\Psi''\Psi^{(3)}\right)}_{\circIV}
  \end{split}
\end{align}
Term $\circI$ is estimated by combining our interpolation estimate
\eqref{eq:interpolation} on the fractional Laplacian with
Gagliardo-Nirenberg
\begin{align*}
  \left|\circI\right| 
  \lesssim \beta_\tau e^{(3\alpha-1)s}
  \norm{\partial_y W}_{L^\infty}^{\frac{1}{3} - \frac{4}{9} \alpha}
  \norm{\partial_y^6 W}_{L^2}^{\frac{2}{3} + \frac{4}{9}\alpha }
  \leqslant \beta_\tau e^{(3\alpha-1)s} M^{\frac{1}{2} +
  \frac{\alpha}{3}}
  \leqslant \ep^{\frac{8}{9}(1-3\alpha)},
\end{align*}
where in the second inequality we have used \eqref{eq:w1_unif_linf}
and \eqref{eq:energy_estimate} and in the final inequality we have
sacrificed a power of $\ep$. Recalling that $h=(\log M)^{-2}$,
the term $\circII$ can be bounded by the bootstrap assumptions
\eqref{eq:w0_boot}-\eqref{eq:w234_boot_near} in the Taylor region
\begin{align*}
  |\circII|
  &\lesssim \ep^{\frac{8}{9}(1-3\alpha)}
    + \log M \ep^{\frac{7}{9}(1-3\alpha)} (h^4 + h^3 + h^2 + h) \\
  &\lesssim \ep^{\frac{8}{9}(1-3\alpha)}
    + (\log M)^{-1} \ep^{\frac{7}{9}(1-3\alpha)}.
\end{align*}
The estimate \eqref{eq:lagrangian_est_2} along
with the explicit property \eqref{eq:psi_jb} on $\Psi^{(5)}$ imply
that
\begin{align*}
  |\circIII| = \left|e^{\sfrac{s}{2}}(\kappa - \dot\xi)\Psi^{(5)}\right| 
  \lesssim \ep^{\frac{9}{10}(1-3\alpha)}
  \leqslant \ep^{\frac{8}{9}(1-3\alpha)}.
\end{align*}
The term $\circIV$ is quadratic in $\Psi$ so
\eqref{eq:bootstrap_assumptions} and \eqref{eq:psi_jb} give
\begin{align*}
  \left|\circIV\right|
  \lesssim \ep^{\frac{1}{2}(1-3\alpha)}
  \leqslant \ep^{\frac{8}{9}(1-3\alpha)}.
\end{align*}

In light of \eqref{eq:fourth_taylor_damping} and our forcing
estimates, we write
\begin{align*}
  \partial_s\partial_y^4\widetilde{W} +
  \frac{1}{4}\partial_y^4\widetilde{W} +
  g_W\partial_y^4\widetilde{W} \geqslant -\beta_\tau \mathcal{F}.
\end{align*}
Composing this inequality with the flowmap $\Phi^{y_0}$ and
integrating gives us
\begin{align}\label{eq:w4_closure}
  \begin{split}
    |\partial_y^4 \widetilde{W}^{y_0}(s)| &\leqslant |\partial_y^4
    \widetilde{W}^{y_0}(s)| e^{-\frac{s}{4}}
    + e^{-\frac{s}{4}}
    \int_{s_0}^s\,e^{\frac{s}{4}} \mathcal{F}^{y_0}\,ds' \\
    & \leqslant \ep^{1-3\alpha} + C (\ep^{\frac{8}{9}(1-3\alpha)} +
    (\log M)^{-1} \ep^{\frac{7}{9}(1-3\alpha)}) \\
    & \leqslant \frac{1}{2}\ep^{\frac{7}{9}},
  \end{split}
\end{align}
where we have used the Lagrangian trajectory lower bound
\eqref{eq:optimal_lower} and the initial condition
\eqref{eq:w234_init_data}. This closes the bootstrap
\eqref{eq:w4_boot_near} for $\partial_y^4\widetilde{W}$.

\paragraph*{(Third Derivative at $0$).}
The third derivative evolves according
to equation \eqref{eq:stable_3rd}, which reads
\begin{align*}
  \bigg(\partial_s + 4 + 4\beta_\tau \partial_yW\bigg)\partial_y^3W
  + g_W\partial_y^4W
  =  -\beta_\tau \left[e^{(3\alpha-1)s}\flap \partial_y^3W +
  3(\partial_y^2 W)^2\right].
\end{align*}
Setting $y=0$ and using the constraints \eqref{eq:w_constraints} we
obtain the identity
\begin{align*}
  \partial_s\partial_y^3W^0
  &= \COL{red}{\underbrace{
    4\dot{\tau}\beta_\tau \partial_y^3 W^0
    }_{\circI}}
    - \COL{blue}{\underbrace{
    \beta_\tau e^{\sfrac{s}{2}}(\kappa-\dot{\xi})\partial_y^4W^0
    }_{\circII}}
    -\COL{orange}{\underbrace{
    \beta_\tau e^{(3\alpha-1)s}\flap [\partial_y^3W]^0
    }_{\circIII}}.
\end{align*}
Using our $\beta_\tau$ estimate \eqref{eq:beta_tau_bound}, the fact
that $|\partial_y^3W^0(s)|\leqslant 7$, and the $\dot{\tau}$ bootstrap
closure \eqref{eq:tau_dot_decay}, we find
\begin{align*}
	|\COL{red}{\circI}| \leqslant 28 e^{\frac{8}{9}(3\alpha-1)s}.
\end{align*}
Next we apply the estimates \eqref{eq:lagrangian_est_2} and
\eqref{eq:beta_tau_bound}, as well as our $L^\infty$ control
\eqref{eq:4th_dv_linf} on $\partial_y^4W$ to find
\begin{align*}
  |\COL{blue}{\circII}|
  \lesssim 2M^{\frac{1}{2}}e^{\frac{9}{10}(3\alpha-1)s}
  \leqslant e^{\frac{8}{9}(3\alpha-1)s}
\end{align*}
upon taking $\ep$ sufficiently small. Finally, we estimate
\begin{align*}
  |\COL{orange}{\circIII}|
  \lesssim 2M^{\frac{1}{3}(1+\alpha)}
  \left(1-\alpha\right)e^{(3\alpha-1)s}
  \leqslant e^{\frac{8}{9}(3\alpha-1)s}
\end{align*}
by applying our interpolation estimate \eqref{eq:interpolation} for
the fractional Laplacian, \eqref{eq:beta_tau_bound}, our $L^\infty$
bounds \eqref{eq:3rd_dv_linf}, \eqref{eq:4th_dv_linf}, and taking $\ep$
sufficiently small.

We record the following fact
\begin{align}\label{eq:w3_time_decay}
  |\partial_s\partial_y^3 W(0,s)| \leqslant
  3e^{\frac{8}{9}(3\alpha-1)s} \leqslant 3\ep^{\frac{8}{9}(1-3\alpha)},
\end{align}
which we will need to use later in our proof.

Recalling that $\partial_y^3W(0,s_0)=6$ and applying the Fundamental
theorem of calculus to the estimate above, we see that
\begin{align}\label{eq:third_der_diff_small}
  \begin{aligned}
    |\partial_y^3W(0,s)-6|
    &\leqslant
    \int_{-\ep}^{t}\,3\ep^{\frac{8}{9}(1-3\alpha)}\,dt' \leqslant
    3\ep^{\frac{8}{9}(1-3\alpha)}(T_*+\ep) \leqslant
    3\ep^{(\frac{8}{9}+\frac{8}{9})(1-3\alpha)}+3\ep^{(1+\frac{8}{9})(1-3\alpha)} \\
    &\leqslant
    6\ep^{\frac{16}{9}(1-3\alpha)}.
  \end{aligned}
\end{align}
Taking $\ep$ small closes the bootstrap \eqref{eq:w3_boot}.

\paragraph*{(Zeroth, First,  Second, and Third Derivative on $0
  \leqslant |y|\leqslant h$).}
By the fundamental theorem of calculus, our bootstrap
\eqref{eq:w3_boot}, and the estimate \eqref{eq:w4_closure},
\begin{align*}
  |\partial_y^3 \widetilde{W}|
  &= \left|\partial_y^3 \widetilde{W}(0,s) +
    \int_{0}^y\,\partial_y^4\widetilde{W}(y',s)\,dy' \right| \\
  & \leqslant 6\ep^{\frac{9}{10}(1-3\alpha)}
    + \frac{1}{2} \ep^{\frac{7}{9}(1-3\alpha)}h \\
  & \leqslant \frac{1}{2} \left( \ep^{\frac{8}{9}(1-3\alpha)}
    + \log M \ep^{\frac{7}{9}(1-3\alpha)} \right)h,
\end{align*}
which closes \eqref{eq:w3_boot_near} on
$\partial_y^3\widetilde{W}$. Recall that our constraints
\eqref{eq:w_constraints} at the origin are
\begin{align*}
	\partial^j_y \widetilde{W}(0,s) = 0
\end{align*}
for $j = 0, 1, 2$. We can iterative apply the fundamental theorem of
calculus to estimate
\begin{align*}
  |\partial_y^j \widetilde{W}|
  &= \left|\partial_y^j\widetilde{W}(0,s) +
    \int_{0}^y\,\partial_y^{j+1} \widetilde{W}(y',s)\,dy'\right|
    \leqslant \frac{1}{2} \left(\ep^{\frac{8}{9}(1-3\alpha)}
    + \log M \ep^{\frac{7}{9}(1-3\alpha)}\right)h^{4-j},
\end{align*}
which closes the bootstraps \eqref{eq:w0_boot_near},
\eqref{eq:w1_boot_near}, and \eqref{eq:w2_boot_near}.

\subsection{Outside the Taylor Region}
To close the bootstraps on the rest of $\R$ we will utilize the
trajectory framework presented in Appendix
\ref{app:transport_estimates}.

\paragraph*{(Zeroth Derivative on $h\leqslant |y| < \infty$).}
We consider the weighted variable
$V:=\jb{x}^{-\sfrac{1}{3}}\widetilde{W}$. The evolution equation for
$V$ is obtained from \eqref{eq:diff} after a short
computation
\begin{align}\label{eq:w0_v_equation}
  \Big(\partial_s - \underbrace{\frac{1}{2} + \beta_\tau\Psi' +
  \frac{g_W}{3}y\jb{y}^{-2}}_{\mathcal{D}(y)}\Big)V + g_WV' = -
  \beta_\tau\jb{y}^{-\frac{1}{3}}F_{\widetilde{W}},
\end{align}
where the forcing is given below in \eqref{eq:Wtilde_forcing}.
Closing the bootstrap \eqref{eq:w0_boot_far} now amounts to showing
that $|V(x,s)|<\ep^{q}$ uniformly in self similar time, for
$h\leqslant |x| <\infty$. From Appendix \ref{app:transport_estimates},
if we can control the damping from below and the forcing from above on
our region $h\leqslant |x| <\infty$, then we can establish global (in
time) estimates by carefully following the trajectories through our
region.

We begin by bounding our damping below by \eqref{eq:beta_tau_bound}, \eqref{eq:psi_jb}, \eqref{eq:lagrangian_est_2}
and our bootstrap assumption \eqref{eq:w0_boot_far},
\begin{align*}
  \mathcal{D}(y)
  &\geqslant \COL{purple}{\frac{1}{2}} -
    \COL{red}{(1+2\ep^{\frac{1}{2}(1-3\alpha)})\jb{y}^{-2/3}}
    - \frac{1}{3}\left(\COL{blue}{\beta_\tau W}
    +\COL{orange}{\beta_\tau e^{s/2}(\kappa-\dot{\xi})}
    + \COL{purple}{\frac{3}{2}y}\right)y\jb{y}^{-2} \\
  &\geqslant
    \COL{purple}{\frac{1}{2}\jb{y}^{-2}}
    - \COL{red}{(1+2\ep^{\frac{1}{2}(1-3\alpha)})\jb{y}^{-2/3}}
    - \COL{blue}{\frac{2}{3}(1+\ep^{q})
    y\jb{y}^{-5/3}}
    -\COL{orange}{\ep^{\frac{1}{2}(1-3\alpha)}|y|\jb{y}^{-2}} \\
  &\geqslant -
    (2+3\ep^{\frac{1}{2}(1-3\alpha)}+\frac{2}{3}\vara)\jb{y}^{-\sfrac{2}{3}}
  \\
  &\geqslant -3\jb{y}^{-2/3}.
\end{align*}
This inequality holds upon taking $\ep$ sufficiently small.

Let $s_*$ be the first time that the Trajectory $\Phi^{x_0}$ enters
the region $h\leqslant |y|< \infty$. By composing the damping
with the lower bound of the trajectories and integrating in time we
find that
\begin{align}\label{eq:w0_damping}
  \begin{split}
  \int_{s_*}^s\,\jb{\Phi^{y_0}(s')}^{-\frac{2}{3}}\,ds'
  &\leqslant
  \int_{s_*}^\infty\,\jb{|y_0|e^{\frac{2}{5}(s-s_*)}}^{-\sfrac{2}{3}}\,ds'
    \leqslant \frac{13}{2}\log\left(\frac{1}{h}\right),
  \end{split}
\end{align}
which is the same estimate obtained in \cite{buckmaster_2019_1}. Using
\eqref{eq:w0_damping}, we set
\begin{align*}
  \lambda_{\mathcal{D}} :=
  e^{-\int_{s_0}^s\,\mathcal{D}(\Phi^{y_0}(s')\,ds'} \leqslant
  e^{\frac{39}{2}\log\left(\frac{1}{h}\right)}
  = h^{-\frac{39}{2}}.
\end{align*}

Next we bound the forcing above in $L^\infty$. Recall that the forcing
for equation \eqref{eq:w0_v_equation} is given by \eqref{eq:diff}
\begin{align}\label{eq:Wtilde_forcing}
  F_{\widetilde{W}} =
  \COL{red}{\underbrace{e^{-\frac{s}{2}}\dot{\kappa}}_{\circI}}
  + \COL{blue}{\underbrace{e^{(3\alpha-1)s}\flap W}_{\circII}}
  + \COL{orange}{\underbrace{\Psi'(\dot{\tau}\Psi)}_{\circIII}}
  +
  \COL{purple}{\underbrace{
  \Psi'e^{\frac{s}{2}}(\kappa-\dot{\xi})}_{\circIV}}.
\end{align}
We also note that
$\jb{y}^{-\sfrac{1}{3}}\circ\Phi^{y_0}(s)\leqslant
e^{-\frac{1}{2}(s-s_0)}$ as a consequence of our lower bound
\eqref{eq:optimal_lower} on the trajectories.

We estimate term by term. For the first term we use identity
\eqref{eq:xi_dot_closed}, the interpolation inequality
\eqref{eq:interpolation}, \eqref{eq:beta_tau_bound},
\eqref{eq:lagrangian_est_1}, and take $\ep$ sufficiently small
\begin{align*}
  \beta_\tau|\jb{\Phi^{y_0}(s)}^{-\sfrac{1}{3}}\COL{red}{\circI}|
  &\lesssim
    \beta_\tau\jb{e^{\frac{3}{2}(s-s_*)}}^{-\sfrac{1}{3}}e^{(3\alpha-1)s}
    \left[\flap
    W^0 - \frac{\flap \partial_y^2 W^0}{\partial_y^3 W^0}\right] \\
  &\lesssim \beta_\tau\ep^{\frac{1}{2}}e^{(3\alpha-\frac{3}{2})s}\left(
    \norm{W}^{1-2\alpha}_{L^\infty}\norm{\partial_yW}^{2\alpha}_{L^\infty}
    + \frac{1}{5}\norm{\partial_y^2W}^{1-2\alpha}_{L^\infty}
    \norm{\partial_y^3W}^{2\alpha}_{L^\infty}
    \right) \\
  &\lesssim \ep^{\frac{1}{2}}e^{(3\alpha-\frac{3}{2})s}
    \left(e^{\frac{s}{2}-\alpha s} +
    \frac{1}{5}M^{\frac{1}{6}-\frac{\alpha}{3}}\right) \\
  &\leqslant \frac{1}{4}e^{(2\alpha-1)s}.
\end{align*}
Note that the $e^{\sfrac{s}{5}}$ lower bound
\eqref{eq:sub_optimal_lower} on the trajectories is insufficient to
close the above estimate.

For the next term we use our interpolation estimate
\eqref{eq:interpolation}, \eqref{eq:lagrangian_est_1},
\eqref{eq:beta_tau_bound}, and sacrifice powers of $\ep$ to get
\begin{align*}
  \beta_\tau|\jb{\Phi^{y_0}(s)}^{-\sfrac{1}{3}}\COL{blue}{\circII}|
  \lesssim \ep^{\sfrac{1}{2}}e^{(3\alpha-\frac{3}{2})s}
    \norm{W}_{L^\infty}^{1-2\alpha}
    \norm{\partial_yW}_{L^\infty}^{2\alpha}
  \leqslant \frac{1}{4}e^{(2\alpha-1)s}.
\end{align*}
We estimate the third term by using \eqref{eq:psi_jb} twice,
\eqref{eq:beta_tau_bound}, our bootstrap assumption
\eqref{eq:bootstrap_assumptions} on $\dot{\tau}$, and taking $\ep$
small
\begin{align*}
  \beta_\tau
  |\jb{\Phi^{y_0}(s)}^{-\sfrac{1}{3}}\COL{orange}{\circIII}|
  \lesssim \dot{\tau}\ep^{\sfrac{1}{2}}e^{-\frac{s}{2}}
    \jb{\Phi^{y_0}(s)}^{-\frac{1}{3}}
  \leqslant \frac{1}{4}e^{-s}.
\end{align*}
For the final term we use \eqref{eq:lagrangian_est_2},
\eqref{eq:psi_jb}, and take $\ep$ small to find
\begin{align*}
  \beta_\tau| \jb{\Phi^{y_0}(s)}^{-\sfrac{1}{3}}\COL{purple}{\circIV}|
  \leqslant
    2\ep^{\sfrac{1}{2}}\frac{e^{(3\alpha-\frac{3}{2})s}}{5}
    \norm{\partial_y^2W}_{L^\infty}^{1-2\alpha}\norm{\partial_y^3
    W}_{L^\infty}^{2\alpha}\jb{\Phi^{y_0}(s)}^{-\sfrac{2}{3}}
  \leqslant \frac{1}{4}e^{(3\alpha-\frac{5}{2})s}.
\end{align*}
Together these four estimates give
\begin{align}\label{eq:w0_forcing}
  |\beta_\tau\jb{y}^{-1/3}F_{\widetilde{W}}| \leqslant e^{(2\alpha-1)s}.
\end{align}

Since $\alpha < \sfrac{1}{3}$, we have temporal decay in our forcing
terms. Inserting the damping estimate \eqref{eq:w0_damping} and the
forcing estimate \eqref{eq:w0_forcing} into
\eqref{eq:general_v_solution} gives
\begin{align*}
  |V^{y_0}(s)|
  &\lesssim \lambda_D|V(y_0)| + \lambda_D
    \int_{s_*}^s\,e^{-(1-2\alpha)s'}\,ds'
  \\
  &\lesssim \lambda_D\left(|V(y_0)| +
    e^{(2\alpha-1)s_*}\right)
\end{align*}
As in the discussion in Appendix \ref{app:transport_estimates}, we
have two cases; either $s_*=s_0$, and $h\leqslant y_0 < \infty$,
or $s_*>s_0$ and $y_0=h$. In the first case we use our initial
data assumption \eqref{eq:w0_init_data} to find that
\begin{align*}
  |V^{y_0}(s)| \lesssim \lambda_D\left(\ep^{1-3\alpha}\jb{y_0}^{0} +
  \ep^{-(1-2\alpha)s_0}\right)
  \leqslant
  \ep^{\frac{1}{2}(1-2\alpha)}.
\end{align*}
In the second case we apply our bootstrap \eqref{eq:w0_boot_near} for
the region $0\leqslant |y| \leqslant h$ and estimate
\begin{align*}
  |V^{y_0}(s)| \lesssim
  \lambda_D\left(
  \left( \ep^{\frac{8}{9}(1-3\alpha)}
  + \log M \ep^{\frac{7}{9}(1-3\alpha)} \right)h^4
  + \ep^{-(1-2\alpha)s_*}\right).
\end{align*}
In both cases we may take $\ep$ sufficiently small to get
\begin{align*}
  |V^{y_0}(s)| \leqslant
  \frac{1}{2}\ep^{q}.
\end{align*}
This closes the zeroth derivative bootstrap \eqref{eq:w0_boot} over
the region $h \leqslant |y| < \infty$.

\paragraph*{(First Derivative on $h\leqslant |y| \leqslant e^{ms}$).}
We use the same strategy used to close the zeroth derivative
bootstrap. Making the weighted change of variables
$V := \jb{y}^{\sfrac{2}{3}}\partial_y\widetilde{W}$, a short
computation using \eqref{eq:diff_deriv} yields the following evolution
equation for $V$
\begin{align}\label{eq:v1_evolution}
  \partial_sV + \Big(\underbrace{1 +
  \beta_\tau(\partial_y\widetilde{W} + 2\Psi') -
  \frac{2}{3}y\jb{y}^{-2}g_W}_{\mathcal{D}(y)}\Big)V + g_W\partial_yV =
  -\beta_\tau\jb{y}^{\sfrac{2}{3}}F_{\partial_y\widetilde{W}}
\end{align}
where $F_{\widetilde{W}'}$ is the forcing term from
\eqref{eq:diff_deriv} which is given below in
\eqref{eq:w1_forcing}. Closing the bootstrap \eqref{eq:w1_boot_middle}
now amounts to showing that $|V(x,s)|< \varb$ uniformly in $s$ for all
$h\leqslant |x| \leqslant e^{\vare s}$.

We begin by bounding the damping below. Using
\eqref{eq:beta_tau_bound}, \eqref{eq:psi_jb}, the bootstraps
\eqref{eq:w0_boot} and \eqref{eq:w1_boot}, and
\eqref{eq:lagrangian_est_2} we find
\begin{align*}
  \mathcal{D}(y)
  &= \COL{red}{-1} - \COL{orange}{\beta_\tau}
    (\COL{blue}{\partial_y\widetilde{W}} +
    2\COL{purple}{\Psi'}) +
    \frac{2}{3}y\jb{y}^{-2}\left(\COL{red}{\frac{3}{2}y} +
    \COL{orange}{\beta_\tau}\left(\COL{teal}{W} +
    \COL{cyan}{e^{\sfrac{s}{2}}(\kappa-\dot{\xi})}\right)\right) \\
  &\geqslant \COL{red}{-\jb{y}^{-2}} -
    \COL{orange}{(1+2\ep^{\frac{1}{2}(1-3\alpha)})}
    \left(
    \COL{blue}{\varb\jb{y}^{-\sfrac{2}{3}}}
    + 2\COL{purple}{\jb{y}^{-\sfrac{2}{3}}} \right. \\
  &\qquad\qquad\qquad\qquad
    \left.+ \COL{teal}{\frac{2}{3}|y|\jb{y}^{-\sfrac{5}{3}}}
    +
    \COL{cyan}{\ep^{\frac{1}{2}(1-3\alpha)}\frac{2}{3}(1+)|y|\jb{y}^{-2}} 
    \right) \\
  &\geqslant -\jb{y}^{-2} - 5\jb{y}^{-\sfrac{2}{3}} \\
  &\geqslant -6\jb{y}^{-\sfrac{2}{3}}.
\end{align*}
Our estimate \eqref{eq:w0_damping} still applies to the current
damping lower bound for the first derivative. Therefore we have
\begin{align}\label{eq:w1_lambda}
  \lambda_{\mathcal{D}} :=
  e^{-\int_{s_0}^s\,\mathcal{D}(\Phi^{y_0}(s'))\,ds'} \leqslant
  e^{39\log(\frac{1}{h})} = h^{-39}.
\end{align}

The forcing for $V$ in this case is given by
\begin{align}\label{eq:w1_forcing}
  F_{\partial_y\widetilde{W}} :=
  \COL{red}{\underbrace{e^{(3\alpha-1)s}\flap \partial_yW}_{\circI}}
  + \COL{blue}{\underbrace{e^{s/2}(\kappa-\dot{\xi})\Psi''}_{\circII}}
  + \COL{orange}{\underbrace{\widetilde{W}\Psi''}_{\circIII}}
  + \COL{purple}{\underbrace{\dot{\tau}\Psi\Psi''}_{\circIV}}
  + \COL{teal}{\underbrace{(\Psi')^2}_{\circV}}
\end{align}
We bound the first term using \eqref{eq:beta_tau_bound}, the
interpolation estimate \eqref{eq:interpolation}, and by taking $\ep$
sufficiently small.
\begin{align*}
  \beta_\tau|\jb{y}^{\sfrac{2}{3}}\COL{red}{\circI}|
  &\lesssim e^{\frac{2}{3}\vare s}e^{(3\alpha-1)s}
    \norm{\partial_yW}_{L^\infty}^{1-2\alpha}
    \norm{\partial_y^2W}_{L^\infty}^{2\alpha} \\
  &\lesssim M^{\frac{1}{6}}e^{(\frac{2}{3}\vare + 3\alpha - 1)s}.
\end{align*}
We choose $m$ such that the exponent in the exponential above remains
positive for all $0 < \alpha < \sfrac{1}{3}$. Remember that we are
aiming to prove that $V$ is bounded by some multiple of
$\ep^{\sfrac{2}{3}(1-3\alpha)}$, which means we must have
$m< \frac{3}{8}(1-3\alpha)$. This informs the choice of $m$ in
\eqref{eq:bootstrap_constants}.

For the second term, we use identity \eqref{eq:xi_dot_closed},
\eqref{eq:psi_jb}, the interpolation estimate
\eqref{eq:interpolation}, the $L^\infty$ control
\eqref{eq:3rd_dv_linf},\eqref{eq:4th_dv_linf}, and then take $\ep$
small to find
\begin{align*}
  \beta_\tau|\jb{y}^{\sfrac{2}{3}}\COL{blue}{\circII}|
  &\leqslant \beta_\tau
    \frac{e^{(3\alpha-1)s}}{\partial_y^3W(0,s)}
    |\flap\partial_y^2 W(0,s)|\jb{y}^{-1} \\
  &\lesssim M^{\frac{1}{6}(1+\alpha)\vard}e^{(3\alpha - 1 - \frac{3}{2})s} .
\end{align*}

We apply the bootstrap \eqref{eq:w0_boot_near} for
$\widetilde{W}$, use \eqref{eq:psi_jb}, and estimate the trajectories
with \eqref{eq:sub_optimal_lower} to get
\begin{align*}
  \beta_\tau
  |\left(\jb{y}^{\sfrac{2}{3}}\COL{orange}{\circIII}\right)
  \circ \Phi^{y_0}(s)|
  &\leqslant \beta_\tau
    \ep^{\frac{1}{2}q}
    \jb{\Phi^{y_0}(s)}^{-\sfrac{2}{3}}
  \\
  &\lesssim
    \ep^{\frac{1}{2}q}e^{-s}.
\end{align*}
Next, \eqref{eq:psi_jb}, \eqref{eq:beta_tau_bound}, and taking $\ep$
small yields
\begin{align*}
  \beta_\tau|\left(\jb{y}^{\sfrac{2}{3}}\COL{purple}{\circIV}\right)
  \circ \Phi^{y_0}(s)|
  \leqslant \beta_\tau\dot{\tau}\jb{\Phi^{y_0}(s)}^{-\sfrac{4}{3}}
  \lesssim \ep^{\frac{8}{9}(1-3\alpha)}e^{-2s}.
\end{align*}
Finally, \eqref{eq:psi_jb}, \eqref{eq:beta_tau_bound}, and
\eqref{eq:mod_deriv_boot_assump} imply that
\begin{align*}
  \beta_\tau|\left(\jb{y}^{\sfrac{2}{3}}\COL{teal}{\circV}\right)
  \circ\Phi^{y_0}(s)| \leqslant
  \beta_\tau\jb{\Phi^{y_0}(s)}^{-\sfrac{2}{3}} \lesssim e^{-s}.
\end{align*}
To summarize, we have obtained the following forcing bound
\begin{align}\label{eq:w1_forcing_est}
  |F_{\partial_y\widetilde{W}}\circ\Phi^{y_0}(s)| \lesssim e^{-\ell s},
\end{align}
with $\ell$ as in \eqref{eq:bootstrap_constants}.

We once again apply the trajectory framework from Appendix
\ref{app:transport_estimates} to our equation \eqref{eq:v1_evolution}
for $V$. Plugging our damping estimate \eqref{eq:w1_lambda} and our
forcing estimate \eqref{eq:w1_forcing_est} into
\eqref{eq:general_v_solution} we obtain
\begin{align*}
  |V^{y_0}(s)|
  &\lesssim h^{-39}|V(y_0)| + \lambda_D\int_{s_*}^s\,e^{-\ell s'}\,ds' \\
  &\lesssim h^{-39}\left(|V(y_0)| + e^{-\ell s_*}\right)
\end{align*}
As in our closure of the zeroth derivative bootstrap and the
discussion of our transport framework from Appendix
\ref{app:transport_estimates}, we have two cases. In the first case,
$s_*=s_0$ and $h\leqslant y_0 \leqslant \ep^{-m}$ and we
use our initial data assumption \eqref{eq:w1_init_data} to find that
\begin{align*}
  |V^{y_0}(s)|\lesssim h^{-39}\left(\varf\jb{y_0}^{-\sfrac{2}{3}}+
  \ep^{\ell}\right) \lesssim M^{20}\left(\varf + \ep^{\ell}\right)
  \leqslant \frac{3}{4}\ep^{\frac{3}{4}\ell},
\end{align*}
upon choosing $\ep$ sufficiently small.

In the second case we have that $s_*>s_0=-\log\ep$, and
\eqref{eq:w1_boot} means $V^{y_0}(s_*)=V(h)$. This gives
\begin{align*}
  |V^{y_0}(s)|\lesssim
  M^{20}\left(\vara h^2\jb{h}^{-2}+\ep^{\ell}\right) \leqslant
  \frac{3}{4}\ep^{\frac{3}{4}\ell}
\end{align*}
(since $\ep^{\frac{3}{4}\ell} < \vara$ for all $\alpha$). Thus we
have closed the first derivative bootstrap \eqref{eq:w1_boot} on the
region $h\leqslant |y| \leqslant e^{m s}$.

\paragraph*{(First Derivative on $e^{m s}\leqslant |x|<\infty$).}
We prove a short fact about the temporal decay of the fractional
Laplacian along trajectories.

\begin{lemma}\label{lem:forcing_decay_far}
  Suppose $e^{ms} \leqslant |y| < \infty$. Then
  \begin{align*}
    \flap[ \partial_yW](\Phi^{y}(s)) \lesssim
    e^{-s}
  \end{align*}
  for all $s$ sufficiently large.
\end{lemma}

\begin{proof}
  We use the singular integral representation \eqref{eq:singular_int}
  for $\flap$. Let
  $\flap_{[a,b]} u := C_\alpha\int_{a\leqslant |\eta|\leqslant
    b}\,\cdots\,d\eta$ and decompose the fractional Laplacian as
  follows
  \begin{align*}
    \flap \partial_y W =
    \COL{red}{\underbrace{
    \flap_{[0,h]}\partial_y\widetilde{W}
    }_{\circI}} +
    \COL{blue}{\underbrace{
    \flap_{[h,e^{ms}]}\partial_y\widetilde{W}
    }_{\circII}} +
    \COL{orange}{\underbrace{
    \flap_{[e^{ms},\infty)}\partial_yW
    }_{\circIII}} +
    \COL{purple}{\underbrace{
    \flap_{[0,e^{ms}]}\Psi'
    }_{\circIV}}.
  \end{align*}
  The term $\COL{purple}{\circIV}$ is majorized by
  \eqref{eq:psi_flap_control}, that is
  $|\circIV| \leqslant \jb{y}^{-\sfrac{2}{3}-2\alpha}$. Composing with
  the lower bound \eqref{eq:optimal_lower} on the trajectories shows
  that $|\COL{purple}{\circIV}|\lesssim e^{-(1+3\alpha)s}$.

  We now go term by term and bound using our bootstraps
  \eqref{eq:w1_boot}. Using \eqref{eq:w1_boot_near} for the first term
  we have 
  \begin{align*}
    |\COL{red}{\circI}|
    &\leqslant C_\alpha\int_{0\leqslant|\eta|\leqslant h}
      \frac{e^{-s} + C}{|y-\eta|^{1+2\alpha}}\,d\eta \\
    &\leqslant 2hC_\alpha(e^{-s}+C)|y-h|^{-1-2\alpha},
  \end{align*}
  where $C$ is the constant from \eqref{eq:w1_boot_near}.

  Trajectories will either satisfy $\Phi^{y_0}(s)>e^{ms}$ for
  $|y_0|\geqslant e^{ms_0}$, or $s_*$ is the first time that a
  trajectory enters $e^{ms}\leqslant y < \infty$, in which case
  $|y_*|=e^{ms_*}$. In both cases since $e^{ms_0}>1$ we can apply
  our trajectory lower bound \eqref{eq:optimal_lower}. Since
  $y=\Phi^{y_0}(s) \geqslant |y_0| > e^{ms} > 1 > h$ we have that
  $|y-h|>|y|$ and composing with trajectories gives
  $|y-h| \gtrsim e^{\frac{3}{2}s}$. Therefore 
  \begin{align*}
    |\COL{red}{\circI}| \lesssim
    e^{-\frac{3}{2}(1+2\alpha)s}.
  \end{align*}
  We remark that this estimate required a lower bound on the
  trajectories of at least $e^{\frac{3}{5}s}$.
 
  The integrand becomes singular at $y=e^{ms}$. Since $\partial_yW$
  and $\partial_y^2 W$ are both $L^\infty$, the interpolation
  inequality \eqref{eq:interpolation} guarantees that the integral
  converges. Because we only care about the asymptotic behavior of
  these quantities, and since we will be composing with trajectories,
  we can safely assume that $|x|\geqslant e^{ms}+1$. The same argument
  as before applies and trajectories will still take off like
  $e^{\frac{3}{2}s}$.

  To this end we estimate
  \begin{align*}
    |\COL{blue}{\circII}|
    &\leqslant C_\alpha\int_{h\leqslant|\eta|\leqslant e^{ms}}
      \frac{e^{-s} + \varb \jb{y}^{-\sfrac{2}{3}}}
      {|y-\eta|^{1+2\alpha}}\,d\eta  \\
    &\leqslant 2C_\alpha (e^{-s} +
      \varb\jb{h}^{-\sfrac{2}{3}})(e^{ms}-h)
      |y-e^{ms}|^{1+2\alpha}, \\
    |\COL{orange}{\circIII}|
    &\leqslant 2C_\alpha\int_{e^{ms}\leqslant|\eta|<\infty}
      \frac{e^{-s}}
      {|y-\eta|^{1+2\alpha}}\,d\eta \\
    &= \frac{2C_\alpha}{\alpha}e^{-s}|y-e^{ms}|^{-2\alpha}
  \end{align*}
  since $|y-e^{ms}| > 1$ our lower bound \eqref{eq:optimal_lower}
  applies (with $y_0=\Phi(s_*)-e^{ms_*}$) and
  \begin{align*}
    |\COL{blue}{\circII}^{y_0}(s)| \lesssim
    e^{ms}e^{-\frac{3}{2}(1+2\alpha)s}, \qquad
    |\COL{orange}{\circIII}^{y_0}(s)| \lesssim
    e^{-(1+3\alpha)s}
  \end{align*}
  for $s$ sufficiently large. Collecting the estimates on
  $\COL{red}{\circI},\COL{blue}{\circII}$, and
  $\COL{orange}{\circIII}$ proves the lemma.
\end{proof}

We proceed in the same manner as we did when closing the bootstraps
for $\widetilde{W}$ on $h\leqslant|y|<\infty$ and
$\partial_y\widetilde{W}$ on $h\leqslant |y|\leqslant e^{ms}$. Making
the weighted change of variables
$V:=e^{s}\partial_y W$, a short computation using
\eqref{eq:stable_1st} shows
that $V$ satisfies the equation
\begin{align*}
  (\partial_s + \beta_\tau \partial_y W)V + g_W\partial_yV =
  -\beta_\tau e^{3\alpha s}\flap \partial_y W.
\end{align*}
Closing the bootstrap \eqref{eq:w1_boot_far} now amounts to showing
that $|V|<2$ for all $e^{ms}\leqslant |y|$. We use
\eqref{eq:beta_tau_bound} to bound the damping below
\begin{align*}
  \mathcal{D}(y) = \beta_\tau \partial_y W \geqslant -
  \frac{3}{2}e^{-s}
\end{align*}
In particular the damping is integrable and we obtain
\begin{align*}
  \lambda_{\mathcal{D}} := e^{-\int_{s_*}^s\,\mathcal{D}^{y_0}\,ds'} 
  \leqslant \ep^{\frac{3}{2}}
\end{align*}

Lemma \ref{lem:forcing_decay_far} shows that the
forcing is bounded and we applying our trajectory
framework from Appendix \ref{app:transport_estimates} to obtain
\begin{align*}
  |V^{y_0}(s)|\lesssim \lambda_{\mathcal{D}}|V(y_0)| +
  \lambda_{\mathcal{D}}\int_{s_*}^{s}\,e^{(3\alpha-1)s'}\,ds'.
\end{align*}
We once again have either $s_*=s_0$ or $s_* > s_0$. In the first case
our initial data assumption \eqref{eq:w1_init_far} gives
\begin{align*}
  |V^{y_0}(s)|\lesssim \lambda_{\mathcal{D}}\left(e^{s_0}\ep^2 \varc +
  e^{(3\alpha-1)s_*}\right) \leqslant 1,
\end{align*}
upon $\ep$ sufficiently small. In the
second case our bootstrap on $h\leqslant |y|\leqslant e^{ms}$
\eqref{eq:w1_boot_middle} implies that
\begin{align*}
  |V^{y_0}(s)|\lesssim \lambda_{\mathcal{D}}\left(\jb{e^{ms_*}}^{-2/3}
  + \varb\jb{e^{ms}}^{-2/3} + e^{(3\alpha-1)s_*}\right) \leqslant 1,
\end{align*}
since $e^{-ms_*} \leqslant e^{-ms_0}=\ep^{m}$. The inequality holds
upon choosing $\ep$ sufficiently small. This closes the bootstrap
\eqref{eq:w1_boot_far}.

\section{Proof of Theorem}\label{sec:proof}
We are now ready to prove Theorem \ref{thm:stable_main_thm} and the
Corollary \ref{cor:open_set}. We prove uniform H\"older bounds on our
solution up to the shock time in subsection \ref{sec:holder_bounds},
and asymptotic convergence to a stable Burgers profile $\Psi_\nu$ in
subsection \ref{sec:asymptotic_convergence}. We prove Theorem
\ref{thm:stable_main_thm} in subsection \ref{sec:proof_of_stable_main}
and we conclude with subsection \ref{sec:failure_of_mod} discussing
why our method cannot be extended beyond the range
$0<\alpha<\sfrac{1}{3}$.

\subsection{H\"older Bounds}\label{sec:holder_bounds}
First we note that for all $0\leqslant |y|\leqslant h < \sfrac{1}{2}$,
the bootstrap \eqref{eq:w1_boot_near} gives
\begin{align*}
  |\widetilde{W}| \leqslant \left( \ep^{\frac{8}{9}(1-3\alpha)} + \log
  M \ep^{\frac{7}{9}(1-3\alpha)} \right) h^3\int_0^{|y|}\, 1\,dy =
   C h^4|y| \leqslant C h^4|y|^{1/3}.
\end{align*}
Combining this estimate with the bootstrap \eqref{eq:w0_boot_far} we
get
\begin{align*}
  |\widetilde{W}| &\leqslant
  \begin{cases}
    C h^3|y|^{1/3},&\qquad 0 \leqslant |y|\leqslant h \\
    \ep^{q}\jb{y}^{\sfrac{1}{3}}, &\qquad h \leqslant |y|<\infty
  \end{cases} \\
    &\leqslant |y|^{1/3}.
\end{align*}
We then note that the bound \eqref{eq:psi_0_13} for $\Psi$ implies
that
\begin{align}\label{eq:w_algebraic_bound}
  |W| \leqslant |\widetilde{W}| + |\Psi| \leqslant 2|y|^{1/3}.
\end{align}

From the transformation \eqref{eq:self_similar_w_ansatz} we obtain the
following equality for the H\"older seminorms of $u$ and $W$
\begin{align*}
  [u]_{C^{1/3}}=\sup_{\stackrel{x,z\in\R}{x\ne z}}\frac{|u(x,t)-u(z,t)|}{|x-z|^{1/3}}
  &=
    \sup_{\stackrel{x,z\in\R}{x\ne z}}\frac{e^{-\sfrac{s}{2}}|W(xe^{\sfrac{3s}{2}},s)-W(ze^{\sfrac{3s}{2}},s)|}{(e^{-\sfrac{3s}{2}}|xe^{\sfrac{3s}{2}}-ze^{\sfrac{3s}{2}}|)^{1/3}}
  \\
  &=\sup_{\stackrel{x',z'\in\R}{x'\ne z'}}
    \frac{|W(x',s)-W(z',s)|}{|x'-z'|^{1/3}} = [W]_{C^{1/3}}.
\end{align*}
Furthermore, \eqref{eq:w_algebraic_bound} implies that
$[W]_{C^{1/3}}\leqslant 2$ uniformly in both $x$ and $s$,
and since $u\in L^\infty$ we have that
\begin{align}\label{eq:u_holder}
  \norm{u}_{C^{1/3}} = \norm{u}_{L^\infty} + [u]_{C^{1/3}} \leqslant M
  + 2.
\end{align}
Therefore $u$ is H\"older $1/3$ uniformly in $x$ and $t$. We also
point out that by a similar argument, any H\"older norm larger than
$\sfrac{1}{3}$ cannot be uniformly controlled in time and blows up
when the singularity forms. Indeed, letting $\beta > \sfrac{1}{3}$, a
simple calculation shows that
\begin{align*}
[u]_{C^{\beta}}= e^{\frac{s}{2}(3\beta - 1)}[W]_{C^{\beta}}.
\end{align*}
Hence $u$ is not $C^\beta$ at $T_*$ for any $\beta>\frac{1}{3}$.

\subsection{Asymptotic Convergence to
  Stationary Solution}\label{sec:asymptotic_convergence}
The key observation underpinning our analysis is that whenever
$\flap W$ is bounded the self-similar equation
\eqref{eq:modulated_pde} governing the evolution of $W$ formally
converges to the self-similar Burgers equation
\eqref{eq:burgers_similarity_equation}. This section justifies this
limit rigorously.

We follow the proof outlined by Yang in \cite{yang_2020}, which is
based on the proof in \cite{buckmaster_2019_1}.

\subsubsection{Taylor Expansion}
Set $\nu=\lim_{s\rightarrow \infty}\partial_y^3W(0,s)$; this limit
exists by the fundamental theorem of calculus and the estimate
\eqref{eq:w3_time_decay}. By \eqref{eq:w3_boot}, we know that
$5 \leqslant |\nu| \leqslant 7$. Since $\Psi_\nu$ satisfies the same
constraints at the origin as $\Psi$, namely \eqref{eq:w_constraints}
(c.f. Appendix \ref{app:burgers}), we automatically obtain
$W(0,s)=\Psi_\nu(0)$, $\partial_yW(0,s)=\Psi_\nu'(0)$, and
$\partial_y^2W(0,s)=\Psi_\nu''(0)$ for all $s\geqslant s_0$.

The difference $\widetilde{W}_\nu:=W-\Psi_\nu$ ($\Psi_\nu$ defined in
\eqref{eq:psi_nu}) satisfies the evolution equation
\begin{align}\label{eq:diff_asymp}
  \left(\partial_s - \frac{1}{2} + \Psi_\nu'\right)\widetilde{W}_\nu +
  \left(\frac{3}{2}y + W\right)\partial_y\widetilde{W}_\nu =
  -\beta_\tau F_{\widetilde{W}_\nu},
\end{align}
where the forcing $F_{\widetilde{W}_\nu}$ is given below by
\eqref{eq:asymp_forcing}. We will now show that
$W(y,s)\rightarrow \Psi_\nu(y)$ for all $y\in \R$ as
$s\rightarrow \infty$, that is
\begin{align}\label{eq:convergence_sup}
  \limsup_{s\rightarrow\infty}|\widetilde{W}_\nu(y,s)|=0.
\end{align}
The convergence \eqref{eq:convergence_sup} is trivial when $y_0=0$
since $\Psi_{\nu}(0)=\Psi(0)$ and $W(0,s)=0$ by our constraint
\eqref{eq:w_constraints}.

We consider the Taylor expansion of $\widetilde{W}_\nu$, observing
that this Taylor expansion vanishes to third order as a consequence of
our constraints \eqref{eq:w_constraints}:
\begin{align*}
  \widetilde{W}_\nu(y,s) &= \widetilde{W}_\nu(0,s)
  + y\partial_y\widetilde{W}_\nu(0,s) 
  + \frac{y^2}{2}\partial_y^2\widetilde{W}_\nu(0,s)
  + \frac{y^3}{6}\partial_y^3\widetilde{W}_\nu(0,s)
  + \frac{y^4}{24}\partial_y^4\widetilde{W}_\nu(\eta,s) \\
  &=\frac{y^3}{6}\partial_y^3\widetilde{W}_\nu(0,s)
  + \frac{y^4}{24}\partial_y^4\widetilde{W}_\nu(\eta,s), \qquad
    0<|\eta|<\infty.
\end{align*}
Differentiating \eqref{eq:psi_nu} four times,
applying \eqref{eq:4th_dv_linf}, and taking $\ep$ small gives the following
estimate
\begin{align*}
  \norm{\partial_y^4 \widetilde{W}_\nu}_{L^\infty_{y,s}}
  \leqslant \norm{\Psi_\nu^{(4)}}_{L^\infty} +
  \norm{\partial_y^4W}_{L^\infty_{y,s}} \\
  \leqslant 30\left(\frac{7}{6}\right)^{\frac{3}{2}} +
  M^{\frac{6}{7}\vard} \leqslant 2M^{\frac{6}{7}\vard}.
\end{align*}
We take the absolute value of the Taylor expansion and apply
\eqref{eq:w3_time_decay} to find that
\begin{align}\label{eq:convergergence_taylor}
  |\widetilde{W}_\nu(y,s)| \leqslant
  \frac{1}{6}|y|^3e^{-\frac{1}{4}(3\alpha-1)s} +
  \frac{M^{\frac{1}{2}\vard}}{12}|y|^4
\end{align}
holds for all $y\in \R$.

Now fix any $y_0\in \R$, with $|y_0|>0$, and choose $0<\lambda<1$ such
that $\lambda\leqslant |y_0|\leqslant \lambda^{-1/6}$. From the Taylor
expansion \eqref{eq:convergergence_taylor}, we may choose
\begin{align*}
  s_*=\max\left\{
  \log\left(\left(\frac{6\delta}{|\lambda|^3}
  \right)^{-\frac{4}{3\alpha-1}}\right),s_0\right\}.
\end{align*}
If we then choose some $\delta$ such that $\lambda^4 > \delta > 0$,
then our Taylor expansion \eqref{eq:convergergence_taylor}
collapses to
\begin{align}\label{eq:wnu_taylor}
  |\widetilde{W}_\nu(y,s)| \leqslant
  \delta + \frac{M^{\frac{1}{2}}}{12}|y|^4.
\end{align}
This choice for $\delta$ will become clear below.

\subsubsection{Lagrangian Trajectories}
Consider the Lagrangian flow associated with
\eqref{eq:diff_asymp}, given by
\begin{align}\label{eq:convergence_lagrangian}
  \frac{d}{ds}\Phi^{y_0}
  &= \frac{3}{2}\Phi^{y_0} + W^{y_0}(s).
\end{align}
From the Taylor expansion of $W$ about $y=0$, the mean
value theorem together with the uniform $L^\infty$ bound
\eqref{eq:w1_unif_linf} shows
\begin{align*}
  |W(y,s)|\leqslant |y| \,\text{ for all } \,y\in \R.
\end{align*}
Therefore $\frac{d}{ds}|\Phi^{y_0}|^2\geqslant |\Phi^{y_0}|^2$, which
upon integration yields the lower bound
\begin{align*}
  |\Phi^{y_0}(s)| \geqslant |y_0|e^{\frac{1}{2}(s-s_0)}.
\end{align*}
The same upper bound as before, \eqref{eq:lagrangian_upper}, still
applies to the current trajectories. Thus we have established that for
all $0<y_0<\infty$
\begin{align}\label{eq:asymp_traj_bounds}
  |y_0|e^{\frac{1}{2}(s-s_0)} \leqslant \Phi^{y_0}(s) \leqslant
  \left(|y_0|+\frac{3}{2}C\ep^{-\sfrac{1}{2}}\right)e^{\frac{3}{2}(s-s_0)}.
\end{align}
Note that our new bounds are independent of any fixed parameter $h$,
and indeed hold for any $y_0>0$. This is in contrast to our previous
estimates, where some trajectories near the origin will not take off
at all!

\subsubsection{Forcing Terms}
We prove two technical lemmas to deal with the fractional Laplacian in
the range $\sfrac{1}{4}\leqslant \alpha < \sfrac{1}{3}$.

\begin{lemma}\label{prop:flap_0}
  For $\alpha > 1/6$, we have
  \begin{align}\label{eq:flapW_0}
    | (\flap W)^0 | \lesssim \frac{1}{h^{2\alpha}}.
  \end{align}
\end{lemma}

\begin{proof}
  Let $\chi$ be a smooth cutoff function such that $\chi \equiv 1$ on
  $ 0 < |y| < h/2$, $\chi \equiv 0$ on $h < |y|$, and which smoothly
  interpolates between $1$ and $0$ on $h/2 < |y| < h$ with
  $|\chi'| \leq 4/h$. We decompose
  \begin{align*}
    (\flap W)^0 = (\flap (\chi W))^0 + (\flap ((1-\chi)W))^0
  \end{align*}
  To estimate the first term, we use our interpolation
  \eqref{eq:interpolation} on $\flap$, the bootstraps
  \eqref{eq:w0_boot_near} and \eqref{eq:w1_boot_near}, and
  $|\chi'|\leqslant \sfrac{4}{h}$ to get
  \begin{align*}
    \|\flap (\chi W)^0\|_{L^\infty}
    \lesssim \|\chi W^0\|_{L^\infty}^{1-2\alpha}
    \|(\chi W)'\|_{L^\infty}^{2\alpha}
    \lesssim \frac{1}{h^{2\alpha}}.
  \end{align*}
  To bound the second term, we use \eqref{eq:singular_int} and the
  fact that $\sfrac{2}{3} + 2\alpha > 1$ to obtain
  \begin{align*}
    |\flap [(1-\chi)W]^0|
    \leqslant 2C_\alpha
    \int_{h/2}^\infty\, (1-\chi(\eta))
    \frac{|W(\eta)|}{\eta^{1+2\alpha}}\,dy
    \lesssim \int_{h/2}^\infty\,\eta^{-\frac{2}{3} - 2\alpha}\,dy
    \lesssim h^{\frac{1}{3} - 2\alpha}.
  \end{align*}
  Since $h<1$ the result follows.
\end{proof}

\begin{lemma}\label{prop:flapW_decay}
  For $\alpha \in (0,\sfrac{1}{3})$, we have
  \begin{align}\label{eq:flapW_decay}
    \norm{\flap W}_{L^\infty}
    \lesssim 
    \begin{cases} 
    1 + M e^{(\frac{1}{2}-3\alpha)s} & \alpha \neq \sfrac{1}{6}\\
    M + s & \alpha = \sfrac{1}{6}
    \end{cases}
  \end{align}
\end{lemma}

\begin{proof}
  We decompose the domain of the singular integral representation
  \eqref{eq:singular_int} of $\flap$ into three regions and estimate
  \begin{align*}
    |\flap W(y)| \lesssim
    \int_\mathbb{R}\,
    \frac{|W(y) - W(\eta)|}{|y - \eta|^{1+2\alpha}}\, d\eta
    \lesssim \left( \int_{0 < |\eta-y| < 1}
    + \int_{1 < |\eta-y| < e^{3s/2}}
    + \int_{e^{3s/2} < |\eta-y|}\right)\,
    \frac{|W(y) - W(\eta)|}{|y - \eta|^{1+2\alpha}}\,d\eta.
  \end{align*}
  Around $y$, the mean value theorem and \eqref{eq:w1_unif_linf} give
  $|W(y)-W(\eta)|\leqslant |y-\eta|$. Therefore
  \begin{align*}
    \int_{0<|\eta-y|<1}\,
    \frac{|W(y) - W(\eta)|}{|y - \eta|^{1+2\alpha}} d\eta
    \lesssim \int_{0<|\eta-y|<1}\,\frac{1}{|y-\eta|^{2\alpha}}\,d\eta
    \lesssim 1.
  \end{align*}

  In the second region we use the \textit{a posteriori} H\"older
  $C^{1/3}$ seminorm \eqref{eq:w_algebraic_bound} to obtain
  \begin{align*}
    \int_{1<|\eta-y|<e^{3s/2}}\,
    \frac{|W(y) - W(\eta)|}{|y - \eta|^{1+2\alpha}}\,d\eta 
    \lesssim \int_{1<|\eta-y|<e^{3s/2}}\,
      \frac{|y - \eta|^{1/3}}{|y - \eta|^{1+2\alpha}}\,d\eta
      \lesssim
      \begin{cases}
      1 + e^{(\frac{1}{2} - 3\alpha)s} & \alpha \neq \sfrac{1}{6} \\
      1 + s & \alpha = \sfrac{1}{6}.
      \end{cases}
  \end{align*}
  Finally, we use \eqref{eq:lagrangian_est_1} and
  \eqref{eq:kappa_bound} to estimate the third integral
  \begin{align*}
    \int_{e^{3s/2}<|\eta-y|}\,\frac{|W(y) - W(\eta)|}{|y -
    \eta|^{1+2\alpha}}\,d\eta
    \lesssim M e^{\sfrac{s}{2}}
    \int_{e^{3s/2}<|\eta-y|}\,\frac{1}{|y-\eta|^{1+2\alpha}}\,d\eta
    \lesssim M e^{(\frac{1}{2} - 3\alpha)s}.
  \end{align*}
  The three estimates above prove the lemma.
\end{proof}

We proceed to bound the forcing term in \eqref{eq:diff_asymp}, which
is given by
\begin{align}\label{eq:asymp_forcing}
  F_{\widetilde{W}_\nu} =
  \COL{red}{\underbrace{e^{-\sfrac{s}{2}}\dot{\kappa}}_{\circI}} +
  \COL{blue}{\underbrace{e^{(3\alpha-1)s}\flap W}_{\circII}} +
  \COL{orange}{
  \underbrace{\partial_yWe^{\sfrac{s}{2}}(\kappa-\dot{\xi})}_{\circIII}}
  +
  \COL{purple}{\underbrace{\dot{\tau}W\partial_yW}_{\circIV}}
\end{align}
In what should now feel like a \textit{danse famili\`ere}, we go term
by term bounding the forcing in $L^\infty$.

When
$\alpha < \sfrac{1}{4}$, we use the identity \eqref{eq:xi_dot_closed},
the interpolation estimate \eqref{eq:interpolation}, and the bootstrap
\eqref{eq:w3_boot} to obtain
\begin{align*}
    |\COL{red}{\circI}|
  &\leqslant e^{(3\alpha-1)s}\left(\frac{|\flap
    \partial_y^2W^0(s)|}{|\partial_y^3W^0(s)|}
    + |\flap W^0(s)|\right) \\
  &\lesssim M^{\sfrac{1}{6} + 3\alpha}e^{(2\alpha - \frac{1}{2})s} \\
  & \leqslant e^{(\alpha - \frac{1}{4})s}
\end{align*}
On the other hand, when $\alpha \geq \sfrac{1}{4}$, we use
Lemma \ref{prop:flap_0} to handle the second term, and obtain
\begin{align*}
  |\circI| \lesssim \frac{ M^{\sfrac{1}{6} + 3\alpha}}{h^{2\alpha}}
  e^{(3\alpha-1)s} \leqslant e^{\frac{1}{2}(3\alpha - 1)s}.
\end{align*}
Lemma \ref{prop:flapW_decay} means that the second term is bounded by
\begin{align*}
  |\COL{blue}{\circII}| \lesssim
  e^{(3\alpha-1)s}\left(1 + Me^{(\sfrac{1}{2}-3\alpha)s}\right)
  \leqslant e^{-\frac{1}{4}s},
\end{align*}
upon taking $\ep$ small.

The identity \eqref{eq:xi_dot_closed}, the interpolation estimate
\eqref{eq:interpolation}, the bounds
\eqref{eq:2nd_dv_linf}-\eqref{eq:3rd_dv_linf}, the bootstrap
\eqref{eq:w1_boot}, and the lower bound \eqref{eq:asymp_traj_bounds}
on the trajectories gives us
\begin{align*}
  |\COL{orange}{\circIII}|
  \lesssim
  (1+\varb)\jb{y}^{-\sfrac{2}{3}}e^{(3\alpha-1)s}
  M^{\frac{2}{7}(1-\alpha)\vard}
  \lesssim (1+\varb) M^{\frac{2}{7}(1-\alpha)\vard}
  e^{(3\alpha-\frac{4}{3})s}.
\end{align*}

For the last term, the bootstraps
\eqref{eq:w0_boot}-\eqref{eq:w1_boot}, the lower bound
\eqref{eq:asymp_traj_bounds} on our trajectories,
\eqref{eq:mod_deriv_boot_assump}, and taking $\ep$ small gives us the
bound
\begin{align*}
  |\COL{purple}{\circIV}| \leqslant
  (1+\ep^{\frac{1}{2}q})(1+\varb)
  \jb{y}^{-\frac{1}{3}} \leqslant 2e^{-\frac{1}{6}s}.
\end{align*}
To summarize, we have shown that
\begin{align}\label{eq:tildew_forcing_est}
  |F_{\widetilde{W}_\nu}^{\lambda}(s)| \lesssim e^{-ps}, \qquad p =
  p(\alpha) > 0.
\end{align}

\subsubsection{Putting it all Together}
We set $G(y,s)=e^{-\frac{3}{2}(s-s_*)}\widetilde{W}_{\nu}(y,s)$,
then compose with trajectories \eqref{eq:convergence_lagrangian} and
compute
\begin{align*}
  \frac{d}{ds}G^{\lambda}
  &= -\frac{3}{2}G^{\lambda} +
    e^{-\frac{3}{2}(s-s_*)}\frac{d}{ds}\widetilde{W}_\nu^{\lambda} \\
  &= -\frac{3}{2}G^{\lambda} +
    e^{-\frac{3}{2}(s-s_*)}\left[
    \left(\frac{1}{2}-(\Psi_\nu')^{\lambda}\right)
    \widetilde{W}_\nu^{\lambda}
    -\beta_\tau F_{\widetilde{W}_\nu}^{\lambda}
    \right] \\
  &= \left(-1 - (\Psi_\nu')^{\lambda}\right)G^{\lambda}
    - \beta_\tau e^{-\frac{3}{2}(s-s_*)}F_{\widetilde{W}_\nu}^{\lambda}.
\end{align*}
Note that the damping $1+(\Psi_\nu')^{\lambda}\geqslant 0$ for all
$s\geqslant s_*$, since $\norm{\Psi_\nu'}_{L^\infty}=1$ by
\eqref{eq:psi_nu}. Thus we can apply Gr\"onwall's inequality,
our forcing estimate \eqref{eq:tildew_forcing_est}, and our Taylor
expansion \eqref{eq:wnu_taylor}, and take $s_*$ sufficiently large to
obtain
\begin{align*}
  |G^{\lambda}|
  &\leqslant |G(\lambda,s_*)| +
    \beta_\tau \int_{s_*}^s\,e^{-\frac{3}{2}(s'-s_*)}
    |F_{\widetilde{W}_\nu}^{\lambda}(s')|\,ds' \\
  &\lesssim |\widetilde{W}_\nu^{\lambda}(s_*)| +
    2\int_{s_*}^s\,e^{-\frac{3}{2}(s'-s_*)}e^{-ps'}\,ds'
  \\
  &\lesssim |\widetilde{W}_\nu^{\lambda}(s_*)| +
    e^{-p s_*} \\
  &\leqslant \delta + \frac{M^{\frac{1}{2}\vard}}{12}|\lambda|^4 +
    \delta \\
  &\lesssim M^{\frac{1}{2}\vard}|\lambda|^4.
\end{align*}
For all times
$s_*\leqslant s \leqslant s_* + \frac{7}{3}\log|\lambda|^{-1}$,
by our definition of $G$, we have that
\begin{align*}
  |\widetilde{W}^{y_0}_{\nu}|
  &\lesssim
     M^{\frac{1}{2}\vard}|\lambda|^4e^{\frac{3}{2}(s-s_*)}\\
  &\lesssim M^{\frac{1}{2}\vard}|\lambda|^{1/2}.
\end{align*}
For all $y$ between $\lambda$ and
$\Phi(\lambda,s_*+\frac{7}{3}\log|\lambda|^{-1})$, there exists
$s_*\leqslant s\leqslant s_*+\frac{7}{3}\log|\lambda|^{-1}$ such that
$y=\Phi(\lambda,s)$. Therefore, for any pair $(y,s)$ the previous
estimate gives
\begin{align*}
  |\widetilde{W}_{\nu}(y,s)| \lesssim M^{\frac{1}{2}\vard}|\lambda|^{1/2}.
\end{align*}
By composing with our trajectory lower bound this will cover at least
all $y$ such that
\begin{align*}
  \lambda \leqslant |y| \leqslant
  \lambda e^{\frac{1}{2}(s-s_*)} = \lambda^{-1/6}.
\end{align*}
Now, taking the limit that $s_*\rightarrow \infty$, for all $\lambda
\leqslant |y| \leqslant \lambda^{-1/6}$ we have that
\begin{align*}
  \limsup_{s\rightarrow \infty}|\widetilde{W}_\nu(y,s)|\lesssim
  M^{\frac{1}{2}}|\lambda|^{1/2}.
\end{align*}
Finally, taking $\lambda\rightarrow 0$ proves that for all $y\ne 0$
\begin{align*}
  \limsup_{s\rightarrow \infty}|\widetilde{W}_\nu(y,s)|=0.
\end{align*}
This completes the proof of \eqref{eq:convergence_sup}.

\subsection{Proof of Theorem \ref{thm:stable_main_thm}}
\label{sec:proof_of_stable_main}

We are finally equipped to prove our main theorem, Theorem
\ref{thm:stable_main_thm}.

\begin{enumerate}[label=(\roman*)]
\item \textbf{(Solution is smooth before $T_*$).}
  Recall the uniform
  $L^2$ bound for $\partial_yW$ proven in Lemma
  \ref{lem:uniform_w_1_l2} and the uniform $L^2$ bound of
  $\partial_y^6W$ proven in \eqref{eq:energy_estimate}; interpolation
  via Gagliardo-Nirenberg (Lemma \ref{lem:gn_interpolation}) bounds
  the intermediary $L^2$ norms of $W$ uniformly in self-similar time.
  
  From the transformation \eqref{eq:self_similar_w_ansatz} and by
  differentiating through the $L^2$ norm we obtain the following
  family of identities
  \begin{align*}
    \norm{\partial_x^nu}_{L^2} =
    e^{\left(-\frac{5}{4}+\frac{3}{2}n\right)s}\norm{\partial_y^nW}_{L^2}
  \end{align*}
  relating the $L^2$ norms in physical space to the $L^2$ norms in the
  self-similar space. Note that
  Burgers equation satisfies $L^2$ conservation and the fractional
  term may be dropped from estimates, ensuring that $\norm{u}_{L^2}$ is
  uniformly bounded. These identities along with the uniform
  boundedness of the $L^2$ norms for $\partial_yW$ through
  $\partial_y^6W$ shows that the $H^6$ norm of $u$ remains finite for
  all times prior to $T^*$.

  From the above considerations and the local-in-time existence proven
  in \cite{kiselev_2008, alibaud_2007} we deduce that $u\in
  C([-\ep,\overline{T}]; H^6(\R))$ for any $\overline{T}<T_*$.

\item \textbf{(Blowup location is unique).}  Fix
  $x^\flat \ne x_*$. Differentiating the transformation
  \eqref{eq:self_similar_w_ansatz} gives the identity
  \begin{align*}
    \partial_x u(x^\flat,t) =
    e^{s}\partial_yW\left((x^\flat-\xi(t))e^{\frac{3}{2}s},
    s\right).
  \end{align*}

  Because $x^\flat\ne x_*$ and $\xi(t)\rightarrow x_*$
  as $t\rightarrow T_*$, there exists a time $t^{\flat}$ such that for
  all $t^\flat < t \leqslant T_*$ we can choose
  $1 > \lambda > 0$ such that
  \begin{align*}
    |x^\flat-\xi(t)| > \lambda > 0.
  \end{align*}
  Therefore
  \begin{align*}
    |x^\flat - \xi(t)|e^{\frac{3}{2}s} > \lambda e^{\frac{3}{2}s},
  \end{align*}
  and for all
  \begin{align*}
    s > s_\lambda:=
    \max\{
    \frac{3}{2}\left(\frac{1}{3/2-m}\right)\log\lambda^{-1},
    -\log(T_*-t^\flat)
    \},
  \end{align*}
  we satisfy the inequality
  \begin{align*}
    |x^\flat - \xi(t)|e^{\frac{3}{2}s} > e^{ms}.
  \end{align*}
  We emphasize that the lower bound $s_\lambda$ on times for which the
  above inequality holds grows arbitrarily large as
  $\lambda\rightarrow 0^+$.

  We can apply our bootstrap \eqref{eq:w1_boot_far} to find that for
  all $s$ in this range
  \begin{align*}
    |\partial_yW\left((x^\flat-\xi(t))e^{\frac{3}{2}s},
    s\right)| \lesssim e^{-s},
  \end{align*}
  and hence
  \begin{align*}
    \limsup_{t\rightarrow T_*} |\partial_x u(x^\flat,t)|
    \lesssim 1.
  \end{align*}
  Thus the blowup location is unique.
  
\item \textbf{(Blowup time and location).}  From our definition of
  $\tau$, the blowup time $T_*$ is the unique fixed point
  $\tau(T_*)=T_*$, and in light of \eqref{eq:mod_var_constraints} this
  is equivalent to
  \begin{align*}
    \int_{-\ep}^{T_*}\,(1-\dot{\tau}(t))\,dt = \ep.
  \end{align*}
  Applying the bootstrap closure \eqref{eq:tau_dot_decay} we find that
  \begin{align*}
    \ep = \int_{-\ep}^{T_*}\,(1-\dot{\tau}(t))\,dt \geqslant
    \int_{-\ep}^{T_*}\,1-\frac{3}{4}\ep^{\frac{8}{9}(1-3\alpha)}\,dt,
  \end{align*}
  from which it follows that
  $|T_*|\leqslant \frac{3}{4}\ep^{\frac{8}{9}(1-3\alpha)}$ upon
  choosing $\ep$ sufficiently small.

  Similarly, our definition of $\xi$ together with
  \eqref{eq:mod_var_constraints} gives the condition
  \begin{align*}
    \int_{-\ep}^{T_*}\,\dot{\xi}(t)\,dt = x_*.
  \end{align*}
  Applying the bootstrap closure \eqref{eq:xi_estimate} gives us
  \begin{align*}
    x_*= \int_{-\ep}^{T_*}\,\dot{\xi}(t)\,dt \leqslant
    2M(T_*+\ep) \leqslant 3M\ep.
  \end{align*}
  This proves the claimed bounds on $T_*$ and $y_*$.
  
\item \textbf{(Precise control of $\partial_x u$ at $t=T_*$).}
  $u$ develops a shock (gradient blowup) at $t=T_*$. Recall the
  identity \eqref{eq:self_similar_variables} which, together with
  differentiating the transformation \eqref{eq:self_similar_w_ansatz},
  gives
  \begin{align*}
    \partial_x u(\xi(x),t) =
    \frac{1}{\tau(t)-t}\partial_yW(0,s) \geqslant -\frac{1}{\tau(t)-t}.
  \end{align*}
  
  Next note that for all $-\ep \leqslant t< T_*$ we have that
  \begin{align*}
    \frac{1}{2}\leqslant \frac{\tau(t)-t}{T_*-t}\leqslant 1.
  \end{align*}
  The upper bound is obvious since $\tau(t)$ is monotone
  increasing and $T_*$ is the unique fixed point of $\tau$. The lower
  bound is a consequence of the fact that
  \begin{align*}
    \frac{T_*}{2} \leqslant \tau(t)-\frac{1}{2}t.
  \end{align*}
  This follows from the fact that $\tau(t)-t/2$ is monotone decreasing
  and at time $t=T_*$ the r.h.s. is $T_*/2$. We conclude that
  \begin{align*}
    -\frac{1}{\tau(t)-t} \leqslant
    \partial_x u(\xi(x),t) = -\norm{\partial_x
    u(\cdot,t)}_{L^\infty}
    \leqslant -\frac{1}{2}\frac{1}{\tau(t)-t}.
  \end{align*}
  Thus we have proven the desired behavior of the gradient at the
  singularity.

\item \textbf{($W\rightarrow \Psi$ asymptotically in self-similar
    space).}
  This was proved above in Section \ref{sec:asymptotic_convergence}.
  
\item \textbf{(Shock is H\"older $\sfrac{1}{3}$).}
  This is a consequence of the uniformity of the H\"older bound which
  was shown above in Section \ref{sec:holder_bounds}. Take the limit
  as $s\rightarrow \infty$ to find that $u(y,T_*)\in C^{1/3}$.
\end{enumerate}

\subsection{Proof of Corollary \ref{cor:open_set}}\label{sec:open_set}
We begin by noting that $\kappa_0$ and $\ep$ can be taken in an open
neighborhood since all of our previous arguments require only that
$\ep$ is sufficiently small. From \eqref{eq:w_init_origin}, this
implies that $u_0(0)$ and $\partial_x u_0(0)$ can be taken in an
open set.

Interpolating between \eqref{eq:linf_third_physical} and
\eqref{eq:l2_fifth_physical} yields
$\|\partial_x^4u_0\|_{L^\infty} =
\mathcal{O}(\ep^{-11/2})$. Using this in a Taylor expansion for
$\partial_x^2 u_0$ around the origin, we find that for all $0 <
|\bar{x}|< |x|$
\begin{align*}
\partial_x^2 u_0(x) &= \partial_x^2 u_0(0) + \partial_x^3u_0(0) x 
+ \frac{ \partial_x^4 u_0(\bar{x}) }{2} \bar{x}^2\\
&= \partial_x^2 u_0(0) + 6 \ep^{-4} x + x( \partial_x^3 u_0(0) - 6 \ep^{-4})
+ \mathcal{O}(\ep^{-11/2})\bar{x}^2.
\end{align*}
We will show that we can relax $\partial_x^2 u_0(0) = 0$ and
\eqref{eq:u0_d3_at_0} by assuming instead that
$|\partial_x^2 u_0(0)|$ is sufficiently small, and considering a
small interval around 0 for $x$ on which
$|\partial_x^3 u_0(x) - 6 \ep^{-4} | < 7
\ep^{-\sfrac{4}{9}(5+12\alpha)}$. We can take
$x( \partial_x^3 u_0(0) - 6 \ep^{-4})$ small relative to
$6 \ep^{-4} x$, and similarly for
$\mathcal{O}(\ep^{-11/2})\bar{x}^2$.  Since $ 6 \ep^{-4} x$
dominates, by intermediate value theorem, there exists $x^*$
sufficiently small so $\partial_x^2 u_0(x^*) = 0$. We can
now simply change coordinates $x \mapsto x + x^*$.

The inequalities in
\eqref{eq:u0_init_data}-\eqref{eq:l2_fifth_physical} can be replaced
by strict inequalities by introducing a pre-factor slightly greater
than 1. A sufficiently small $H^6$ perturbation preserves all these
inequalities by slightly enlarging the pre-factor. n particular, by
Sobolev embedding, a small $W^{5,\infty}$ perturbation will satisfy
these inequalities.

\subsection{Stable Modulation Past $\alpha=1/3$}
\label{sec:failure_of_mod} 
Since the stable Burgers profile satisfies $-1 < \Psi'(x) < 0$ for
$x \neq 0$ with $\Psi'(0) = -1$, it follows from
\eqref{eq:singular_int} that $\flap[\Psi'](0) < 0$.  If we consider $W$ to be a perturbation of $\Psi$,
 the modulation constraint \eqref{eq:tau_dot_closed} for $\tau$ amounts to
\begin{align}
  \dot{\tau} \approx e^{(3\alpha-1)s}|\flap[\Psi'](0)|.
\end{align}
When $\alpha \geqslant \sfrac{1}{3}$, the right-hand side is positive
and bounded away from zero for all $s$, so the modulation parameter
$\tau$ diverges as $s\to\infty$.

Oh and Pasqualotto \cite{oh_2021} have recently proven that the
solutions to fractal Burgers equation \eqref{eq:fvbeq} in the range
$\sfrac{1}{3}\leqslant \alpha < \sfrac{1}{2}$ with specifically chosen
initial data converge in an asymptotically self-similar manner to
\textit{unstable} profiles. However these solutions are not stable
under \textit{generic} perturbations.

The blowup criterion supplied by Dong, Du, and Li \cite{dong_2009} is
continuous which means that there is stable singularity formation in
the range $\sfrac{1}{3}\leqslant \alpha < \sfrac{1}{2}$. Our numerics
suggest that this stable blowup profile is not the stable Burgers
profile.

\appendix
\section{Toolbox}\label{app:toolbox}
\subsection{A Framework for Weighted Transport
  Estimates}\label{app:transport_estimates}
We present a framework for doing weighted transport estimates which
was originally introduced in \cite{buckmaster_2019_1}. See also
Section 6.5 of \cite{yang_2020}. We sketch the main ideas here. Our
goal is to bound the $L^\infty$ evolution of a transport-type equation
by using information over some spatial region.

Consider the forced-damped transport equation
\begin{align*}
  \partial_s\mathcal{U} + \mathcal{D}\mathcal{U}
  + g_W\partial_y\mathcal{U} = \mathcal{F}.
\end{align*}
We denote the damping $\mathcal{D}$ and the forcing $\mathcal{F}$. The
advection velocity $g_W$ is defined in \eqref{eq:transport_speed}.

We consider the weighted quantity $V:=\jb{x}^{p}\mathcal{U}$, where
$p$ is some power. It is straightforward to compute the evolution of
$V$
\begin{align}\label{eq:general_v_evolution}
  \partial_sV + \left(\mathcal{D} - 2g_Wpx\jb{x}^{-1}g_W\right)V +
  g_wV = \jb{x}^{p}\mathcal{F}.
\end{align}

Consider the Lagrangian trajectory $\Phi^{x_0}$ obtained above by
solving \eqref{eq:lagrange_traj}, and set $V^{x_0}(s) := V\circ
\Phi^{x_0}(s)$. Composing the evolution
\eqref{eq:general_v_evolution} with $\Phi^{x_0}(s)$ gives 
\begin{align*}
  \frac{d}{ds} V^{x_0} +
  \left(\mathcal{D}\circ\Phi^{x_0} -
  2g_Wp\Phi^{x_0}\jb{\Phi^{x_0}}^{-1}g_W\right)V^{x_0} +
  g_WV^{x_0} = \jb{\Phi^{x_0}}^{p}\mathcal{F}\circ\Phi^{x_0}.
\end{align*}
This equation can be solved via separation of variables to obtain
\begin{align}\label{eq:general_v_solution}
  V^{x_0}(s) = V(x_0)e^{-\int_{s_*}^s\,\mathcal{D}\circ
  \Phi^{x_0}(s')\,ds'} + \int_{s_*}^s\,\mathcal{F}\circ\Phi^{x_0}(s')e^{-\int_{s_*}^{s'}\,\mathcal{D}\circ
  \Phi^{x_0}(s'')\,ds''}\,ds',
\end{align}
where $s_*$ is the first time that the trajectory enters the region in
question.

\subsection{Lemmas}
We make frequent use of the following lemma to bound the fractional
Laplacian in terms of our $L^\infty$ bootstraps.

\begin{lemma}\label{prop:bound_flap}
  \textbf{($L^\infty$ Interpolation of the Fractional Laplacian.)}
  Let $u\in W^{1,\infty}(\R)$ and
  $\alpha \in (0,\sfrac{1}{2})$, then $\flap u \in L^\infty$ and
  satisfies the following estimate
  \begin{align}\label{eq:interpolation}
    \norm{\flap u}_{L^\infty} \lesssim \norm{u}_{L^\infty}^{1-2\alpha}
    \norm{u'}_{L^\infty}^{2\alpha}.
  \end{align}
\end{lemma}

\begin{proof}
  Let $0 < h$ to be chosen later and write the fractional
  Laplacian in its singular integral form \eqref{eq:singular_int}
  \begin{align*}
    \frac{1}{C_\alpha}\flap u
    &=
    \int_{\R\setminus
    \B_h(x)}\,\frac{u(x)-u(y)}{|x-y|^{2\alpha + 1}}\,dy
    + \int_{\B_h(x)}\,\frac{u(x)-u(y)}{|x-y|^{2\alpha +
      1}}\,dy \\
    &=: I_1 + I_2.
  \end{align*}
  For the first integral we estimate
  \begin{align*}
    |I_1| \leqslant \int_{\R\setminus
    \B_h(x)}\,\frac{u(x)-u(y)}{|x-y|^{1+2\alpha}}\,dy
    &\leqslant 4\norm{u}_{L^\infty}\int_{h}^\infty\,
      \frac{1}{y^{1+2\alpha}}\,dy \\
    &\leqslant \frac{1}{\alpha h^{2\alpha}}\norm{u}_{L^\infty}.
  \end{align*}
  For the second integral, the mean value theorem gives
  \begin{align*}
    u'(\eta,t) = \frac{u(x,t)-u(y,t)}{x - y}, \qquad |\eta| \leqslant
    |x-y|.
  \end{align*}
  This implies that
  \begin{align*}
    |I_2| \leqslant \int_{\B_h(x)}\,\frac{u(x)-u(y)}
    {|x-y|^{1+2\alpha}}\,dy
    &= \int_{\B_h(x)}\,\frac{u'(\eta)}{|x-y|^{2\alpha}}\,dy \\
    &\leqslant 2\norm{u'}_{L^\infty}\int_0^h\frac{1}{y^{2\alpha}}\,dy \\
    &= \frac{1}{1-2\alpha}h^{1-2\alpha}\norm{u'}_{L^\infty}.
  \end{align*}
  
  To summarize, we now have
  \begin{align*}
    \frac{1}{C_\alpha}\flap u \leqslant \frac{1}{\alpha
    h^{2\alpha}}\norm{u}_{L^\infty} +
    \frac{1}{1-2\alpha}h^{1-2\alpha}\norm{u'}_{L^\infty}.
  \end{align*}
  Letting $A=\norm{u}_{L^\infty}$,
  $B=\norm{u'}_{L^\infty}$, and $h=A^\beta B^\gamma$, we
  see that the only possible scaling is $\beta=1,\gamma=-1$, i.e. we
  choose $h=\norm{u}_{L^\infty}/\norm{u'}_{L^\infty}$ and
  obtain
  \begin{align*}
    \frac{1}{C_\alpha}\flap u \leqslant
    \left(\frac{1}{\alpha}+\frac{1}{1-2\alpha}\right)
    \norm{u}_{L^\infty}^{1-2\alpha}\norm{u'}_{L^\infty}^{2\alpha}.
  \end{align*}
\end{proof}

We also use the Gagliardo-Nirenberg Interpolation Theorem for $\R$.

\begin{lemma}\label{lem:gn_interpolation}
  \textbf{(Gagliardo-Nirenberg Interpolation).}
  Fix $1\leqslant q, r\leqslant \infty$, $j,m\in \N$ with
  $\sfrac{j}{m}\leqslant \theta\leqslant 1$. If we have the following
  relationship
  \begin{align*}
    \frac{1}{p} = j + \theta\left(\frac{1}{r}-m\right) +
    \frac{1-\theta}{q},
  \end{align*}
  then
  \begin{align*}
    \norm{\partial_x^ju}_{L^p}\lesssim
    \norm{\partial_x^mu}_{L^r}^\theta\norm{u}_{L^q}^{1-\theta}.
  \end{align*}
  The implicit constant depends only on $j,m,r,p,q$, and
  $\theta$ (i.e. is independent of $u$).
  
\end{lemma}

\section{Derivation and Properties of the Self-Similar Burgers
  Profile Family}\label{app:burgers}
Consider the similarity equation for the Burgers equation
\cite{eggers_2015}
\begin{align}\label{eq:burgers_similarity_equation}
  -\frac{1}{2}\Psi + \left(\frac{3}{2}y + \Psi\right)\Psi' = 0.
\end{align}
This is a homogeneous ODE in $\Psi$. It is elementary to obtain the
following implicit family of solutions
\begin{align}\label{eq:burgers_family_solutions}
  x = - \Psi_\nu - \frac{\nu}{6}\Psi^3_\nu,
\end{align}
where $\nu/6$ is the constant of integration. We denote by $\Psi_\nu$
the solution of \eqref{eq:burgers_similarity_equation} corresponding
to a particular choice of $\nu$.

Taking successive (implicit) derivatives we easily obtain the
following properties of $\Psi_\nu$:
\begin{align}
  \Psi_\nu(0)=0, \quad \Psi_\nu'(0)=-1, \quad \Psi_\nu''(0)=0, \quad
  \Psi_\nu'''(0)=\nu, \quad \Psi_\nu^{(2k)}(0)=0, \quad k\in \N.
\end{align}
This shows that the solution family
\eqref{eq:burgers_family_solutions} is parameterized by it's third
derivative at the origin.

We take by convention $\Psi:=\Psi_6$, which has the explicit solution
\begin{align}\label{eq:stable_burg_explicit}
  \Psi(x)=
  \left(-\frac{x}{2}+\left(\frac{1}{27}+\frac{x^2}{4}\right)^{1/2}\right)^{1/3}
  -\left(\frac{x}{2}+\left(\frac{1}{27}+\frac{x^2}{4}\right)^{1/2}\right)^{1/3}.
\end{align}
We use this particular profile throughout our analysis. Furthermore we
can Taylor expand the derivative $\Psi'$ around $x=0$ to obtain
\begin{align}\label{eq:psi_x_taylor}
  \Psi'(x) = -1 + 3x^2 - 15x^4 + \mathcal{O}(x^6).
\end{align}

In fact, once we have obtained the solution for $\Psi$, the solutions
for all of the profiles are obtained through the formula
\begin{align}\label{eq:psi_nu}
  \Psi_\nu(x) = \left(\frac{\nu}{6}\right)^{-1/2}
  \Psi\left(\left(\frac{\nu}{6}\right)^{1/2}x\right).
\end{align}
This can be verified by noting that if $\Psi$ solves
\eqref{eq:burgers_family_solutions} for $\nu=1$, then $\Psi_\nu$
solves the same equation with, excusing the abuse of notation,
$\nu=\nu$.

We note that the explicit form of \textit{any} of the solutions in
the family \eqref{eq:burgers_family_solutions} can be obtained
quite easily using Mathematica or some other solver (or by hand if the
reader is a spiritual medium for Gerolamao Cardano),
since the explicit solution \eqref{eq:stable_burg_explicit} is nothing
more than the solution of the cubic $x=-y-\nu y^3$.

Once an explicit solution is obtained, the reader can easily verify
the following inequalities which (we suggest using Mathematica)
\begin{align}\label{eq:psi_0_13}
  |\Psi(x)| \leqslant |x|^{1/3},
\end{align}
along with the Japanese bracket estimates
\begin{align}\label{eq:psi_jb}
  |\partial_x^i \Psi(x)|\lesssim \jb{x}^{\sfrac{1}{3} - i}, 
\end{align}
for $i = 1,...,5$ and $x>1$. In the case of $i=1,2$ the constant can
be taken to be $1$. The constant for $i=5$, for example, can be taken
as $360$.

The fractional Laplacian applied to the stable profile $\Psi^{(n)}$
behaves exactly as one would expect, namely we have the following
pointwise estimate:

\begin{lemma}\label{lem:decayEst}
  The bounds \eqref{eq:psi_jb} imply
  \begin{align}\label{eq:psi_flap_control}
    |\flap \Psi'(x) | \lesssim \jb{x}^{-\sfrac{2}{3}-2\alpha}.
  \end{align}
\end{lemma}

\begin{proof}
  Without loss of generality, suppose $x > 0$. 
  We decompose the integral form \eqref{eq:singular_int} of $\flap$ by
  \begin{align*}
    (-\Delta)^\alpha \Psi'(x) = C_\alpha
    \left(\int_{-\infty}^{-x} + \int_{-x}^{x/2} + \int_{x/2}^{x}
    + \int_{x}^{2x} + \int_{2x}^{\infty} \right)\,\frac{ \Psi'(x) -
    \Psi'(y) }{|x-y|^{1+2\alpha}} \,dy.
  \end{align*}
  The first integral can be bounded as
  \begin{align*}
    \int_{-\infty}^{-x}\,\frac{|\Psi'(x) -
    \Psi'(y)|}{|x-y|^{1+2\alpha}} \,dy
    \lesssim \jb{x}^{-\sfrac{2}{3}}
    \int_{-\infty}^{-x} \frac{1}{|x-y|^{1+2\alpha}} \, dy \lesssim
    \jb{x}^{-\sfrac{2}{3}-2\alpha}.
  \end{align*}
  The second integral can be bounded with
  \begin{align*}
    \int_{-x}^{x/2}\,\frac{|\Psi'(x) -
    \Psi'(y)|}{|x-y|^{1+2\alpha}} \,dy
    \lesssim \int_{-x}^{x/2} \frac{x^{-2/3} +
    |y|^{-2/3}}{x^{1+2\alpha}} \, dy \lesssim x^{-\sfrac{2}{3} - 2\alpha}
    + \frac{1}{x^{1+2\alpha}} \int_{-x}^{x/2} y^{-\sfrac{2}{3}} \, dy
    \lesssim \jb{x}^{-\sfrac{2}{3}-2\alpha}.
  \end{align*}
  To bound the third integral, we use the Holder
  seminorm bound
  \begin{align*} | \Psi' |_{C^{3\alpha}(x/2,x)} \leqslant \|
    \Psi' \|_{L^{\infty}(x/2,x)}^{1-3\alpha} \| \Psi''
    \|_{L^{\infty}(x/2,x)}^{3\alpha} \lesssim ( \jb{x}^{-\sfrac{2}{3}}
    )^{1-3\alpha} ( \jb{x}^{-\sfrac{5}{3}} )^{3\alpha} =
    \jb{x}^{-\sfrac{2}{3} - 3\alpha}.
  \end{align*}
  Using this seminorm bound, we obtain that the third
  integral can be estimated as
  \begin{align*} \lesssim | \Psi' |_{C^{3\alpha}(x/2,x)}
    \int_{x/2}^x \frac{1}{|x-y|^{1-\alpha}} \, dy \lesssim
    \jb{x}^{-\sfrac{2}{3} - 3\alpha} x^\alpha \lesssim
    \jb{x}^{-\sfrac{2}{3}-2\alpha}.
  \end{align*}
  The fourth integral is estimated similarly to the second integral,
  and the fifth integral is estimated similarly to the first.
\end{proof}

\section{Evolution Equations For Derivatives and Differences}
\label{app:evolution_derivs}
We record the equations which are used above in the course of our
proof. The equations are written in the form
$(\partial_s + D)f + g_Wf_y = Ff$ where $D$ is the damping and $F$ the
forcing.

We have the following differentiated forms of equation
\eqref{eq:modulated_pde}

\begin{align}\label{eq:stable_1st}
  \bigg(\partial_s +1 + \beta_\tau \partial_y W\bigg)\partial_yW
  + g_W\partial^2_y W
  =  - \beta_\tau e^{(3\alpha-1)s} \flap \partial_yW
\end{align}
\begin{align}\label{eq:stable_2nd}
  \left(\partial_s + \frac{5}{2} + 3\beta_\tau \partial_yW\right)
  \partial_y^2W
  + g_W\partial^3_yW
  =  - \beta_\tau e^{(3\alpha-1)s} \flap \partial_y^2W
\end{align}
\begin{align}\label{eq:stable_3rd}
  \bigg(\partial_s + 4 + 4\beta_\tau \partial_yW\bigg)\partial_y^3W
  + g_W\partial_y^4W
  =  -\beta_\tau \left[e^{(3\alpha-1)s}\flap \partial_y^3W +
  3(\partial_y^2 W)^2\right]
\end{align}
\begin{align}\label{eq:stable_6th}
  \bigg(\partial_s + \frac{11}{2} + 7 \beta_\tau\partial_yW\bigg)\partial_y^6W
  + g_W\partial_y^7W
  = -\beta_\tau\left[e^{(3\alpha-1)s} \flap \partial_y^6W
  + 35\partial_y^3W\partial_y^4W + 21\partial_y^2W \partial_y^5 W
  \right]
\end{align}
And the equations for the differences $\widetilde{W}$,
$\partial_y\widetilde{W}$, and $\partial_y^4\widetilde{W}$ derivative
\begin{align}
  \begin{aligned}\label{eq:diff}
    \left(\partial_s - \frac{1}{2} +
      \beta_\tau\Psi\right)\widetilde{W} + g_W\partial_y\widetilde{W} =
    -\beta_\tau \left[ e^{-\frac{s}{2}}\dot{\kappa} +
      e^{(3\alpha-1)s}\flap W + \Psi'(\dot{\tau}\Psi +
      e^{\frac{s}{2}}(\kappa-\dot{\xi})\right]
  \end{aligned}
\end{align}

\begin{align}\label{eq:diff_deriv}
  \begin{aligned}
  \left(\partial_s + 1 +
    \beta_\tau \left(\partial_y\widetilde{W}+2\Psi'\right)\right)
  \partial_y\widetilde{W}
  + g_W\partial_y^2\widetilde{W} \\
  = - \beta_\tau\left[e^{(3\alpha-1)s}\flap \partial_yW +
    \left(e^{\sfrac{s}{2}}(\kappa-\dot{\xi}) + \widetilde{W} +
      \dot{\tau}\Psi\right)\Psi'' + \dot{\tau}(\Psi')^2\right]
\end{aligned}
\end{align}

\begin{align}\label{eq:diff_4_deriv}
	\begin{split}
  \left(\partial_s + \frac{11}{2} +
    5\beta_\tau \partial_yW\right)
  \partial_y^4\widetilde{W}
  + g_W\partial_y^4\widetilde{W}
  &= - \beta_\tau\left[e^{(3\alpha-1)s}\flap \partial_y^4W + \left( e^{\sfrac{s}{2}} (\kappa - \dot\xi) 
  + \widetilde{W} +\dot\tau \Psi \right) \partial_y^5 \Psi \right. \\
  &\qquad  + 4 \partial_y \widetilde{W} \partial_y^4 \Psi + 8 \partial_y^2 \widetilde{W} \partial_y^3 \Psi
  + 10 \partial_y^3 \widetilde{W} \partial_y^2 \Psi  \\
  &\qquad \left. + 11 \partial_y^2 \widetilde{W} \partial_y^3 \widetilde{W} + 5 \dot\tau \partial_y \Psi \partial_y^4 \Psi + 10 \dot \tau \partial_y^2 \Psi \partial_y^3 \Psi
  \right]
  \end{split}
\end{align}

\subsection*{Acknowledgments}
We thank S. Shkoller for proposing this problem and for offering
valuable feedback on our work. 

\bibliographystyle{alpha}
\bibliography{bib}

\Addresses

\end{document}